\documentclass{article}

\title{Extensions of the Furstenberg--S\'ark\"ozy theorem via the arithmetic level-$d$ inequality}
\author{
Carlo Francisco E.~Adajar$^{1}$ \and
Rishika Agrawal$^{1}$ \and
Mukul Rai Choudhuri$^{1}$ \and
Chian Yeong Chuah$^{1}$ \and
Steve Fan$^{1}$ \and
Swaroop Hegde$^{1}$ \and
Andrew Lott$^{1,2}$ \and
Krishnamohan Nandakumar$^{1}$ \and
Nagendar Reddy Ponagandla$^{1,2}$
}
\date{}
\usepackage[margin=1in]{geometry} 
\usepackage{amsmath,amssymb,amsfonts,mathrsfs}
\usepackage{amsthm}
\usepackage{microtype}
\usepackage{indentfirst}
\usepackage{graphicx}
\usepackage{pdfsync}
\usepackage{comment}
\usepackage{enumitem}
\usepackage{esint}
\usepackage[colorlinks=true,
            linkcolor=blue,
            urlcolor=blue,
            citecolor=blue]{hyperref}

\usepackage{tikz}
\usetikzlibrary{arrows.meta,positioning,calc}

\newtheorem{thm}{Theorem}[section]

\newtheorem{lem}[thm]{Lemma}
\newtheorem{prop}[thm]{Proposition}
\theoremstyle{definition}

\newtheorem{thmB}{Theorem}
\newtheorem{thmA}{Theorem}

\newtheorem*{thmB'}{Theorem B$^\prime$}
\newtheorem*{thmC'}{Theorem C$^\prime$}

\theoremstyle{definition}
\newtheorem{rem}[thm]{Remark}
\theoremstyle{remark}

\newcommand{\md}[1]{\ensuremath{\,(\operatorname{mod}\, #1)}} 
\newcommand{\mdsub}[1]{\ensuremath{(\mbox{\scriptsize mod}\, #1)}} 

\numberwithin{equation}{section}

\newcommand{\R}{\mathbb{R}}
\newcommand{\N}{\mathbb{N}}
\newcommand{\Z}{\mathbb{Z}}
\newcommand{\1}{\mathbf{1}}

\begin{document}

\maketitle

\begin{center}
$^{1}$University of Georgia \\
$^{2}$ HUN-REN Alfr\'ed R\'enyi Institute of Mathematics (Erd\H{o}s Center)
\end{center}

\begin{abstract}
Green and Sawhney recently obtained a quasipolynomial bound in the Furstenberg--S\'ark\"ozy theorem for square differences by proving an ``arithmetic level-\(d\)'' inequality, thereby yielding a greatly improved density increment scheme. We apply their method to treat general intersective polynomials $h\in\mathbb{Z}[x]$. In particular, let 
\[
D(h(\N),X):= \max\bigl\{|A|:\ A\subseteq [1,X]\cap\N\ \text{and}\ (A-A)\cap h(\N)\subseteq\{0\}\bigr\}.
\]
We prove that for every $0<\mu<1/2$ there are constants $c_0, X_{\text{min}}>0$ depending on $h$ and $\mu$ such that for every $X>X_{\text{min}}$,
\[D(h(\N), X)\leq Xe^{-c_0(\log X)^\mu}.\]
This is the best quantitative upper bound presently known for sets lacking intersective polynomial differences, improving upon the work of Arala \cite{Ara23}. In order to achieve the admissible exponent range $0<\mu<1/2$, we use sieve methods to develop novel exponential sum estimates in the style of Rice, and we use the ``random sparsification'' procedure of Green and Sawhney.

\end{abstract}
\tableofcontents

\section{Introduction}

For a set $F\subseteq \mathbb Z$ and $X>1$, write
\[
D(F,X):= \max\bigl\{|A|:\ A\subseteq [1,X]\cap\N\ \text{and}\ (A-A)\cap F\subseteq\{0\}\bigr\}.
\]
Here, we refer to $F$ as the set of \emph{forbidden differences}. The Furstenberg--S\'ark\"ozy theorem asserts that $D(\{n^2:n\in\mathbb N\},X)=o(X)$, answering a question of Lov\'asz.
This was proved independently by Furstenberg~\cite{Fur77} using ergodic theory and by S\'ark\"ozy~\cite{sarkozy-i,Sar78} using the Hardy--Littlewood circle method and the density increment method of Roth; S\'ark\"ozy's proof gives the quantitative bound
\[
D(\{n^2:n\in\mathbb N\},X) \ll X(\log\log X)^{2/3}(\log X)^{-1/3}.
\]
Pintz--Steiger--Szemer\'edi~\cite{pintz-steiger-szemeredi} introduced a more elaborate double iteration strategy to obtain a bound of the form
\[
D(\{n^2:n\in\mathbb N\},X) \ll X(\log X)^{-c_1\log\log\log\log X}
\]
with $c_1=1/12$, and Bloom--Maynard~\cite{bloom-maynard} later strengthened this to
\[
D(\{n^2:n\in\mathbb N\},X) \ll X(\log X)^{-c_2\log\log\log X}
\]
where $c_2>0$ is an absolute constant.
Very recently, Green and Sawhney~\cite{GreenSawhney2024FSv1, GreenSawhney2025FSv2} developed a new ``arithmetic level-$d$'' inequality (building on hypercontractive ideas of Keller--Lifshitz--Marcus~\cite{KLM23} and Keevash--Lifshitz--Long--Minzer~\cite{kllm}) and pushed the density increment strategy to obtain the following quasipolynomial bound
\begin{equation}\label{V2}
D(\{n^2:n\in\mathbb N\},X) \ll X\exp\bigl(-c_3(\log X)^{1/2}\bigr),
\end{equation}
for some constant $c_3>0$.

We record the arithmetic level-$d$ inequality as in \cite{GreenSawhney2025FSv2}.

\begin{thmA}[Green--Sawhney]\label{level d}
Set $C_0 := 2^{13}$. Let $\alpha \in (0,1/2)$ and let $X \ge 1$ be parameters with $\alpha > 2X^{-1/2}$. Let $\mathcal{Q}$ be a set of pairwise coprime positive integers such that $\max_{q \in \mathcal{Q}} q \le X^{1/(32\log(1/\alpha))}$. Let $1 \le d\le 2^{-7}\log(1/\alpha)$ be an integer. Let $f : [X] \rightarrow \mathbb{C}$ be a function such that $|f(x)|\le 1$ for all $x$. Then either
\begin{equation}\label{main41}\sum_{\substack{S \subseteq \mathcal{Q} \\ |S| = d}}\sum_{\substack{a \md{\prod_{q \in S} q}  \\ q \in S \Rightarrow q \nmid a}}\Big|\widehat{f}\Big( \frac{a}{\prod_{q \in S} q} \Big)\Big|^2 \le \alpha^2 X^2 \Big(\frac{C_0\log(1/\alpha)}{d}\Big)^d,\end{equation}
or else for some set $S \subseteq \mathcal{Q}$ with $1 \le |S| \le 2 \log(1/\alpha)$, and for some $r \in \Z$,
the average of $|f(x)|$ on the progression $P = \{ x \in [X] : x \equiv r \md{\prod_{q \in S} q}\}$ is greater than $2^{|S|} \alpha$.
\end{thmA}

A natural generalization of the Furstenberg--S\'ark\"ozy theorem is to replace squares by the image of a polynomial.
If $h\in\mathbb Z[x]$ is nonconstant, then a necessary condition for a bound of the shape $D(h(\N),X)=o(X)$ is that $h(\mathbb Z)$ contain a nonzero multiple of every $q\in\mathbb N$, in which case we say that $h$ is \emph{intersective}\footnote{This definition makes sense for integer-valued polynomials in $\mathbb{Q}[x]$ as well. Given a nonconstant integer-valued polynomial $h\in\mathbb{Q}[x]$, there exists $m\in\N$ such that $f:=mh\in\Z[x]$. It is clear that if $h$ is intersective, then so is $f$. Moreover, we have $D(\{h(n)\colon n\in\N\},X)\le D(\{f(n)\colon n\in\N\},mX)$.}.
Kamae and Mend\`es France~\cite{KM78} showed that intersectivity is also sufficient for the qualitative conclusion $D(h(\N),X)=o(X)$.
On the quantitative side, Lucier~\cite{Lucier06} proved the first general density-decay bound for intersective polynomials: if $h$ has degree $\deg h=k\ge 2$, then
\[
D(h(\N),X) \ll_h X\frac{(\log\log X)^{\rho/(k-1)}}{(\log X)^{1/(k-1)}},
\]
where 
\[
\rho =
\begin{cases}
3, & k=2,\\
2, & k\ge 3.
\end{cases}
\]
Stronger decay of Pintz--Steiger--Szemer\'edi type was previously available in special cases (e.g.\ for monomials~\cite{BPPS94} and for intersective quadratics~\cite{HLR13}), and Rice~\cite{Ric19} fully extended this to all intersective polynomials, obtaining bounds of the form
\[
D(h(\N),X) \ll_{h,c_5} X(\log X)^{-c_5\log\log\log\log X},
\]
for some constant $c_5>0$ depending on $k$. In order to prove this bound, Rice developed a method of achieving square root cancellation for local Gauss sums, which was necessary to perform the double iteration strategy of Pintz--Steiger--Szemer\'edi. More recently, Arala~\cite{Ara23} used the exponential sum estimates of Rice to generalize the Bloom--Maynard proof to the full class of intersective polynomials, yielding the bound 
\[
D(h(\N),X) \ll_{h} X(\log X)^{-c_6\log\log\log X},
\]
for some constant $c_6>0$ depending on $k$.

The objective of this paper is to prove the following analogue of \eqref{V2} for general intersective polynomials. 
\begin{thm}\label{main}
Let $0<\mu<1/2$. Given any intersective polynomial $h\in \Z[x]$ of degree at least 2, there are constants $c_0, X_{\text{min}}>0$ depending on $h$ and $\mu$ such that 
\begin{equation}\label{mainineq}
D(h(\N), X)\leq Xe^{-c_0(\log X)^{\mu}}
\end{equation}
whenever $X>X_{\text{min}}$. 
\end{thm}

In contrast to the square case, the natural density increment iteration for intersective polynomials necessarily introduces a family of auxiliary polynomials whose coefficients evolve from step to step. As a result, one must establish circle method estimates uniformly across all auxiliary polynomials that arise, while simultaneously controlling the growth of the auxiliary polynomial coefficients in order to ensure that these estimates remain valid at every stage of the iteration. The difficulty is that the  density increment scheme and smoothly weighted exponential sum estimates of Green and Sawhney must be carefully adapted to this evolving family of auxiliary polynomials. In addition, the required auxiliary polynomial exponential sum estimates rely on a sieve method introduced by Rice \cite{Ric19}. In this method, the summation variable is restricted to a sifted set obtained by excluding integers $n$ for which $h_\ell'(n)$ is divisible by certain powers of small primes. This restriction makes it possible to obtain near-square-root cancellation in the associated complete sums.

In order to achieve the desired exponent range $0<\mu<1/2$, we had to introduce new ideas beyond the existing arguments of Rice and Green--Sawhney. The main new ingredient is an application of the fundamental lemma of sieve theory that gives a sieve estimate uniform on short intervals. We then partition the exponential sum into intervals on which the smoothly weighted oscillatory factor has bounded total variation. Since the error in the sieve estimate on each interval is proportional to its length, partial summation yields the required global error after the estimates for the individual intervals are summed. The use of short intervals and partial summation in this way goes back to the classical work of Vinogradov on exponential sums over the primes; see \cite[Theorems 2a--2b]{Vinogradov54}.

To use this short-interval estimate in the major arc analysis, it must also hold after imposing a congruence condition $n\equiv b\pmod q$, where $q$ is the denominator of the relevant rational approximation. The required sieve asymptotic can fail for an arbitrary $q$. We therefore introduce complete moduli and show that $q$ can be enlarged to a complete modulus by a factor bounded in terms of $h$. Because this enlargement changes the complete sum appearing in the major arc estimate, we also prove a corresponding complete-sum estimate of the same strength as Rice's estimate. 

We are able to use a simplified version of the ``random sparsification" procedure of Green and Sawhney.  In particular, we aim for the exponent range $0<\mu<1/2$ as opposed to $\mu=1/2$, and for this reason there is slack in the random sparsification argument which allows us to cut many of the steps. Achieving $\mu=1/2$ in Theorem \ref{main} seems hopeless with current technology, as Rice's exponential sum estimates yield \emph{near} square root cancellation as opposed to square root cancellation for the complete sum.  

The project has gone through several drafts, each improving the admissible range for the exponent in Theorem \ref{main}. Here is a summary of the previous admissible ranges and what we did to achieve them.

\begin{itemize}
\item We achieved the range $0<\mu<1/(k(k-1)+2)$ with $k=\deg h$ by using the proof framework of the first version of Green--Sawhney along with classical polynomial exponential sum estimates. 

\item We achieved the range $0<\mu<1/4$ by developing smoothly weighted versions of Rice's exponential sum estimates \cite{Ric19}, giving the bound stated in version $1$ of this paper on arXiv \cite{v1}. 

\item We achieved the range $0<\mu<1/3$ by adapting the random sparsification procedure found in version $2$ of Green--Sawhney. At this stage we were unable to achieve the desired range $0<\mu<1/2$ because the error term in the available Brun sieve calculation was too large at the scale required by the iteration.

\item In the present paper, we achieve the range $0<\mu<1/2$ by proving an improved major arc estimate via the short-interval strategy described above.
\end{itemize}

There remains a substantial gap between our upper bound in Theorem \ref{main} and the best available lower bounds for $D(h(\N),X)$ when $h\in\Z[x]$ is intersective (see \cite{DoyleRice2021}). Indeed, all known lower bounds in this setting are of the form $X^{c}$ with some constant $0<c<1$. The greedy construction yields $c=1-\frac{1}{\deg h}$, and improvements beyond this are currently only known in special cases, such as monomials and certain polynomials $h\in \Z[x]$ which are divisible by $x^2$ (see \cite{RuzsaSquares1984,Lewko,Younis2019,Wessel2020}).

\section{Notation}
\begin{itemize}
\item We write $A = O(B)\ \text{or}\ A \ll B$
if there exists an absolute constant $C>0$ such that $|A(X)|\le C\,B(X)$ for all sufficiently large $X$. We write $A\gg B$ if $B\ll A$, and we write $A\asymp B$ if both $A\ll B$ and $A\gg B$ are true. When the implied constant is allowed to depend on certain parameters, say $r$ for instance, we write $A\ll_r B$ (equivalently $A=O_r(B)$), $A\gg_r B$ if $B\ll_r A$, and $A\asymp_r B$ if both $A\ll_r B$ and $A\gg_r B$ occur. 

\item We write $A = o(B)$ if $\displaystyle\lim_{X\to\infty}\frac{A(X)}{B(X)}= 0$.

\item For $X>0$, we use $[X]$ to denote the set $\{n\in \N: n\leq X\}$. 

\item We write $\1_{S}$ for the characteristic function of the set $S$.

\item We write $\mathbb{T}:=\R/\Z$ equipped with normalized Haar
measure $d\alpha$ and we define \[
\lVert x \rVert_{\mathbb{T}}:= \min_{n\in \Z}|x-n|.
\]

\item As is standard in number theory, we let $e(t)=e^{2\pi it}$.

\item For $f\in \ell^1(\Z)$, its \emph{Fourier transform} is the function $\widehat f:\mathbb{T}\to\mathbb{C}$
defined by
\[
\widehat f(\alpha) := \sum_{n\in\Z} f(n)\,e(-n\alpha).
\]
When a function is initially defined on a finite interval such as $[X]$, we regard it as extended by zero to $\Z$ when taking its Fourier transform.
\item For $v\in L^1(\mathbb{R}),$ we use the Fourier transform convention
\[\widehat{v}(t):=\int_{\mathbb{R}}v(u)e(-tu)\ du.\]
\item For a nonempty finite set $S$ and a function $g:S\to \mathbb{C}$, we use the averaging notation 
\[\displaystyle \mathbb{E}_{x\in S}g(x):=\frac{1}{|S|}\sum_{x\in S}g(x).\]

\item We use $\Re x$ to denote the real part of $x\in 
\mathbb{C}$.
\item Throughout the paper, we shall fix an intersective polynomial $h\in \Z[x]$ of degree $k$.

\end{itemize}

\section{An overview of the proof}
The rest of the paper will be devoted to the proof of Theorem \ref{main}. For the benefit of the reader, we outline the proof below. 

\textbf{Section} \ref{intersectivepolynomials}. We introduce the basic theory of intersective polynomials and auxiliary polynomials. The most important part of this section is the auxiliary polynomial inheritance proposition (Proposition \ref{inheritance}), which is the key fact that allows us to perform density increment arguments while passing between auxiliary polynomials. These classical ideas go back to the work of Lucier \cite{Lucier06}.

\textbf{Section} \ref{densincsec}. We state the density increment scheme without proof, and we show how it implies Theorem \ref{main}. Certain aspects of this section match what Green and Sawhney do, but our argument is longer and more complicated because we must track the growth of the auxiliary-polynomial parameter $\ell$, and hence of the coefficients, throughout the iteration.

\textbf{Section} \ref{rice-sieve}. We introduce the ideas of Rice. In particular, for an exponential sum with an auxiliary polynomial phase, it is more effective to take the sum over a special sifted subset of the interval rather than over the entire interval. This set is obtained by excluding integers $n$ for which $h_\ell'(n)$ vanishes modulo specified powers of small primes, and the resulting complete sums exhibit near-square-root cancellation in the range relevant to our argument. This is significantly better than the cancellation available from a standard polynomial complete-sum estimate. The strength of the complete-sum estimate is one of the main factors determining the admissible exponent range in the main theorem.

\textbf{Section} \ref{completemoduli}. We introduce the notion of a complete modulus, which is needed for the sieve calculations involving this sifted set. More precisely, the required asymptotic for the elements of the sifted set lying both in a fixed residue class modulo $q$ and in a short interval need not hold for an arbitrary modulus; it holds when $q$ satisfies the completeness condition. We show that every relevant denominator can be enlarged to a complete modulus by a factor bounded in terms of $h$. We also prove the corresponding complete-sum estimate needed after replacing a denominator by its completion. Appendix \ref{sievecalc} proves the short-interval sieve estimate by applying the fundamental lemma of sieve theory in the form stated by Tao \cite{TaoSieveNotes}.

\textbf{Section} \ref{weylestimate}. We first introduce the smooth weight with rapid Fourier decay used by Green and Sawhney. This weight allows us to define unusually narrow major arcs, which is critical for obtaining an effective density increment scheme from the arithmetic level-$d$ inequality. We then apply Rice's Weyl estimate to prove that if the smoothly weighted exponential sum with auxiliary polynomial phase over the sifted set is large at a frequency $\theta$, then $\theta$ has a rational approximation with a small denominator.

\textbf{Section} \ref{majorarcestimate}. We apply the sieve calculation to obtain a major arc estimate for these exponential sums. To obtain a sufficiently small error term, we partition the range of summation into disjoint short intervals on which the smoothly weighted oscillatory factor has bounded total variation. Since the sieve error on each interval is proportional to its length, partial summation on these intervals yields a favorable error after summing over the partition. This strategy goes back to Vinogradov's classical work on exponential sums over the primes; see \cite[Theorems 2a--2b]{Vinogradov54}. This is also where complete moduli enter the argument. If $a/q$ is a rational approximation to the frequency, the denominator $q$ need not be complete. We therefore replace $a/q$ by an equal fraction with complete denominator and use the specialized complete-sum estimate from Section \ref{completemoduli}. The resulting major arc estimate combines the near-square-root cancellation available in the relevant range with the rapid Fourier decay of the smooth weight.

\textbf{Section} \ref{inifouma}. We begin by defining the major and minor arcs, and then prove the minor arc estimate by combining the Weyl and major arc estimates. We argue contrapositively. If the exponential sum is large at a frequency $\theta$, the Weyl estimate supplies a rational approximation $a/q$ which is sufficiently good for us to apply the major arc estimate. The major arc estimate contains both a complete-sum factor depending on $q$ and the Fourier transform of the smooth weight, which decays rapidly as $\theta$ moves away from $a/q$. The complete-sum estimate forces $q$ to be sufficiently small, while the rapid Fourier decay forces $\theta$ to be sufficiently close to $a/q$. Thus $\theta$ belongs to one of the major arcs. This rapid decay allows the major arcs to be sufficiently narrow for the arithmetic level-$d$ inequality to yield quantitatively effective results.

We also establish the connection between physical-space information and Fourier-space information. Unless the physical-space argument has already supplied a density increment, the absence of auxiliary-polynomial differences forces the Fourier transform of $\1_A$ to have substantial $L^2$ mass near rationals with small denominators. The exponential sum estimates described above yield the strong $q^{-1/(2+\epsilon)}$-weighted Fourier-mass estimate needed in the final argument.

\textbf{Section} \ref{classical-density-increment} and \textbf{Section} \ref{applylevd}. We complete the proof of the density increment scheme. The first section uses a large Fourier coefficient and the classical density increment argument when the density is bounded below. In the second, we apply the arithmetic level-$d$ inequality and use the simplified random sparsification strategy to convert the remaining Fourier mass for $\1_A$ into a strong density increment on a sufficiently long progression.

Here is a graph illustrating the main logical dependencies in the proof.

\begin{center}
\begin{tikzpicture}[
  x=1cm,y=1cm,
  every node/.style={font=\scriptsize, align=center},
  box/.style={draw, rounded corners=2pt, thick, inner sep=3.2pt, text width=4.25cm, minimum height=0.94cm, fill=gray!8},
  input/.style={box},
  mid/.style={box},
  prop/.style={box, very thick, text width=4.95cm, minimum height=1.0cm},
  final/.style={box, very thick, text width=4.9cm, minimum height=1.08cm},
  arr/.style={-{Stealth[length=2.0mm,width=1.45mm]}, thick, draw=black, rounded corners=3pt}
]

\node[mid] (rice) at (0,19.1) {Rice exponential-sum estimates \\ Theorem~\ref{rice-gauss-minor}};
\node[mid] (complete) at (5.5,19.1) {Sieve/Complete-modulus machinery\\ Proposition.~\ref{ricebrunsum2}, Lemmas~\ref{fundamental-sieve-lemma}, \ref{bounded-completion-gauss}};
\node[mid] (leveld) at (11.0,19.1) {Arithmetic level-$d$ inequality \\ Theorem~\ref{level d}};

\node[mid] (ratapprox) at (0,16.35) {Rational approximation\\ from large $\widehat g$\\ Lemmas~\ref{Weyl}--\ref{Mukul}};
\node[mid] (major) at (5.5,16.35) {Major arc estimate\\ Lemma~\ref{Rishika}};

\node[mid] (minor) at (0,13.60) {Minor arc estimate\\ Lemma~\ref{lem:minor_arc_estimate}};
\node[mid] (physical) at (5.5,13.60) {Physical-space estimate\\ Lemma~\ref{physest}};

\node[mid] (initial) at (5.5,10.70) {Classical Fourier mass\\ Lemma~\ref{InitialMassDichotomy}};

\node[mid] (star) at (2.15,8.00) {Classical $(\ast)$ increment\\ Section~\ref{classical-density-increment}};
\node[mid] (levelbranch) at (8.85,8.00) {Random sparsification\\ Section~\ref{applylevd}};

\node[mid] (density) at (5.5,5.35) {Density-increment scheme\\ Proposition~\ref{density-increment-calc}};
\node[mid] (inherit) at (11.0,5.35) {Auxiliary polynomial inheritance\\ Proposition~\ref{inheritance}};

\node[mid] (iteration) at (5.5,2.77) {Iteration\\ Lemma~\ref{lem:GS-propagate} and induction};
\node[mid] (main) at (5.5,0.55) {$D(h(\N),X)\le X e^{-c_0(\log X)^\mu}$ \\ Theorem~\ref{main}};

\draw[arr] (rice) -- (ratapprox);
\draw[arr] (rice.east) -- (complete.west);
\draw[arr] (complete) -- (major);
\draw[arr] (leveld) -- (levelbranch);

\draw[arr] (complete.south) -- ++(0,-0.55) -- ++(2.85,0) |- (physical.east);

\draw[arr] (ratapprox) -- (minor);
\draw[arr] (major.south west) -- (minor.east);
\draw[arr] (major.east) -- ++(1.05,0) |- (initial.east);
\draw[arr] (minor.south) |- (initial.west);
\draw[arr] (physical) -- (initial);

\draw[arr] (initial.south west) -- ++(0,-0.55) -| (star.north);
\draw[arr] (initial.south east) -- ++(0,-0.55) -| (levelbranch.north);

\draw[arr] (star.south) |- (density.west);
\coordinate (densityin) at ([xshift=1.1cm]density.north);
\draw[arr] (levelbranch.south) -- ++(0,-0.55) -| (densityin);

\draw[arr] (density) -- (iteration);
\draw[arr] (inherit.south) |- (iteration.east);
\draw[arr] (iteration) -- (main);

\end{tikzpicture}
\end{center}
\newpage

\section{Intersective polynomials}\label{intersectivepolynomials}

Let $h$ be an intersective polynomial of degree at least $2$. We now
define the auxiliary polynomials as in \cite{RiceThesis}. For each prime
$p$, fix a $p$-adic integer $z_p$ such that $h(z_p)=0$. For each
$\ell\in\N$, let $r_\ell$ be the unique integer in $(-\ell,0]$
satisfying
\[
r_\ell\equiv z_p\pmod{p^{v_p(\ell)}}
\qquad\text{for every prime }p\mid\ell.
\]
The Chinese Remainder Theorem gives the existence and uniqueness of
$r_\ell$, and the defining congruences imply that $\ell\mid h(r_\ell)$.

Write $m_p$ for the multiplicity of $z_p$ as a $p$-adic root of $h$.
Define $\lambda(p):=p^{m_p}$ for every prime $p$, and extend $\lambda$
to $\N$ by complete multiplicativity.

For each $\ell\in\N$, we define the \emph{auxiliary polynomial} $h_\ell$ by
\begin{equation}\label{auxipoldef}
h_\ell(n)=\frac{h(r_\ell+\ell n)}{\lambda(\ell)}.
\end{equation}
The discussion following \cite[Definition~6.2]{RiceThesis} shows that $\lambda(\ell)$ divides every coefficient of $h(r_\ell+\ell x)$, and hence $h_\ell\in\Z[x]$.

The following proposition, which is  \cite[Proposition 6.3]{RiceThesis}, gives an important relationship and is the key property driving the iteration argument. 
\begin{prop}[Rice]\label{inheritance}
Let $\ell,q\in \N$ and $x\in\Z$. If $A\subseteq \N$, $(A-A)\cap h_\ell(\N)\subseteq\{0\}$, and
\[
A'\subseteq \{y\in \N:x+\lambda(q)y\in A\},
\]
then
\[
(A'-A')\cap h_{q\ell}(\N)\subseteq\{0\}.
\]
\end{prop}

\begin{rem}\label{hl bounded}
    We also note here that each coefficient of $h_\ell(n)$ is bounded above in absolute value by $R_h\ell^{k-1}$, where $R_h>0$ is a constant depending only on $h$. 
\end{rem}
This follows by expanding $h(r_\ell+\ell n)$ with repeated applications of the binomial theorem and using the bounds $|r_\ell|<\ell$ and  $\ell\leq \lambda(\ell)\leq \ell^k$.

We import the following lemma from equation (6.1) of \cite{RiceThesis}, which will help simplify the exponential sum estimates. 
\begin{lem}\label{lem:invert_hl}
Let $h\in\mathbb Z[x]$ have degree $k\ge 2$ and positive leading coefficient, and let $h_\ell$
be the auxiliary polynomial. Write
\[
h_\ell(t)=b_k t^k+b_{k-1}t^{k-1}+\cdots+b_0,\qquad b_k>0.
\]
Let $X>0$, and suppose that $Y>0$ satisfies $h_\ell(Y)=X$. 
Then
\begin{equation}\label{eq:Y_close_to_N}
Y = \left( \frac{X}{b_k} \right)^{1/k} + O_h(1).
\end{equation}

\end{lem}

\begin{rem}
We remark that it suffices to consider intersective polynomials $h$ such that
\begin{equation}\label{positivity1}
h(x),h'(x),h''(x)>0
\end{equation}
for every real $x>0$ while proving Theorem \ref{main}. Since $A-A$ is symmetric, replacing a polynomial by its negative does not change the forbidden-difference condition. We may therefore first assume that the polynomial has positive leading coefficient.

Assume that Theorem \ref{main} has been established for all intersective polynomials satisfying \eqref{positivity1}, and let $\Tilde h$ be an arbitrary intersective polynomial of degree at least $2$ with positive leading coefficient. There exists a sufficiently large positive integer $C$, depending on $\Tilde h$, such that
\[
\Tilde h(x),\Tilde h'(x),\Tilde h''(x)>0
\]
for every real $x>C$. Define
\[
h(n):=\Tilde h(n+C).
\]
Then $h$ satisfies \eqref{positivity1} and remains intersective. Moreover,
\[
h(\N)=\big\{\Tilde h(n+C):n\in\N\big\}\subseteq\Tilde h(\N).
\]
Consequently, if $A\subset[X]$ and
\[
(A-A)\cap\Tilde h(\N)\subseteq\{0\},
\]
then also
\[
(A-A)\cap h(\N)\subseteq\{0\}.
\]
Applying \eqref{mainineq} to $h$ proves the desired bound for $\Tilde h$.
\end{rem}

Thus, we may assume throughout the paper that $h$ satisfies \eqref{positivity1}. Note that $r_\ell+\ell t>0$ for every real $t\ge1$, since $r_\ell \in (-\ell,0]$. Therefore, by our definition of $h_\ell$, \eqref{positivity1} implies 
\begin{equation}\label{positivity2}
h_\ell(t),h_\ell'(t),h_\ell''(t)>0
\end{equation}
for every real $t\ge1$.

\section{The iteration argument}\label{densincsec}
The goal of this section is to state the density increment result, and then apply it to obtain Theorem \ref{main}.  In \cite{GreenSawhney2024FSv1}, Green and Sawhney use strong induction to prove the main theorem, but in our case, a proof by strong induction is more difficult because we need to track the auxiliary polynomial coefficients. 

The density-increment scheme below is proved for each fixed
$\epsilon>0$. For such an $\epsilon$, set
\[
F(x):=(\log x)^{1/(2+2\epsilon)}.
\]
Whenever $A$ has density $\alpha$, write $L:=\log(1/\alpha)$. In the
proof of Theorem~\ref{main}, after fixing $0<\mu<1/2$, we will choose
$\epsilon$ so that $\mu<1/(2+2\epsilon)$. The goal of
Sections~\ref{rice-sieve}-\ref{applylevd} is to prove the following
density-increment scheme.

\begin{prop}\label{density-increment-calc}
There are constants $c_h, C_h, D_1, d_2>0$ depending only on $h$ and $\epsilon$ with $0 < c_h < \frac{1}{10}$ and $C_h > 10$ such that the following holds. Let $\alpha \in (0, \frac{2}{3}]$, let $X \ge C_h$, and let $\ell\in\N$. Suppose that $A \subseteq [X]$ is a set with density $\alpha$ and no nonzero difference of the form $h_\ell(n)$ with $n\in\N$. Also suppose
\[
\ell\le \exp\big((\log X)^{(2+\epsilon)/(2+2\epsilon)}\big).
\]
Then, if $\alpha>c_h$, the following holds:
\begin{enumerate}
    \item[\textup{($\ast$)}] there is a subprogression of $[X]$ with common difference of the form $\lambda(q)$ for some $q\in\N$ and length $\ge \alpha^{C_h} X$ on which the density of $A$ is at least $\alpha + \alpha^{C_h}$.  
\end{enumerate}
If $\alpha\le c_h$, then at least one of the following options holds:
 \begin{enumerate}
\item[\textup{(1)}] $\alpha  \le \exp(- c_h F(X))$;
\item[\textup{(2)}] for some integer $j$ with $1\le j\le 2L$, there is a subprogression $P$ with common difference of the form $\lambda(q)$ and length 
$\ge X e^{-D_1L^{1+\epsilon}j}$
on which the density of $A$ is at least $\alpha\bigl(1+d_2\bigr)^j$.
\end{enumerate}
\end{prop}

We now state and prove the following lemma, which when combined with Proposition \ref{density-increment-calc}, yields Theorem \ref{main}. 
\begin{lem}[Inductive step]\label{lem:GS-propagate}
Let $c_h,C_h, D_{1}, d_{2}$ be the constants from Proposition~\ref{density-increment-calc}.
Then there exist constants $c_0\in(0,c_h]$ and $X_{\text{min}}\ge C_h$ depending only on $h$ and $\epsilon$ such that the following holds.

Let $\ell\in\N$, let $X\ge X_{\text{min}}$, and let $A\subseteq [X]$ have density $\alpha:=|A|/X\in(0,2/3]$.
Assume $(A-A)\cap h_\ell(\N)\subseteq\{0\}$.
Suppose $\alpha>\exp(-c_hF(X)).$ Let $P\subseteq [X]$ be the progression produced by Proposition~\ref{density-increment-calc}
in one of the options \textup{($\ast$)} or \textup{(2)}. In the case where $P$ is produced by \textup{($\ast$)}, assume also that $\alpha>c_h$.
Write $X':=|P|$, and set $\alpha':=|A\cap P|/|P|$.
If $\alpha'\ \le\ e^{-c_0 F(X')},$
then $\alpha\ \le\ e^{-c_0F(X)}.$
\end{lem}

\begin{proof}
Put
\[
R := \frac{\alpha'}{\alpha} >1,
\qquad
y := \log \left( \frac{X}{X'} \right) \in[0,\log X].
\]

Then,
\[
\alpha= \frac{\alpha'}{R} \le \frac{e^{-c_0F(X')}}{R},
\]
so it suffices to prove
\begin{equation}\label{eq:key-prop}
c_0 ( F(X)-F(X') ) \le \log R.
\end{equation}

Since $0\le y\le \log X$ and $\displaystyle \frac{1}{2+2\epsilon}\leq 1$, we have
\begin{align*}
    c_0(F(X)-F(X') )&= c_0((\log{X})^{1/(2+2\epsilon)}-(\log{X'})^{1/(2+2\epsilon)}) \\
    & = c_0(\log{X})^{1/(2+2\epsilon)}\left(1-\left(1-\frac{y}{\log{X}}\right)^{1/(2+2\epsilon)}\right) \\
     & \leq c_0(\log{X})^{1/(2+2\epsilon)}\left(1-\left(1-\frac{y}{\log{X}}\right)\right)\\
    &=c_0  y(\log X)^{1/(2+2\epsilon)-1}.
\end{align*}
Hence it suffices to check \begin{equation}\label{eq:Fdrop}
    c_0 y(\log X)^{1/(2+2\epsilon)-1}\leq \log R
\end{equation}
in each alternative, using only the corresponding bounds for $R$ and $y$.

\medskip
\noindent\emph{Case \textup{($\ast$)}.}
Here $\alpha> c_h$ and $\alpha'\ge \alpha+\alpha^{C_h}$, so $R\ge 1+\alpha^{C_h-1}\ge 1+c_h^{C_h-1}$, and also $\frac{X'}{X} \ge \alpha^{C_h}\ge {c_h}^{C_h}$, hence $y\ll_{h,\epsilon}1$.
So by choosing some $c_0 \leq c_h$ and $X_{\text{min}}$ large enough, for $X\ge X_{\text{min}}$ we have \eqref{eq:Fdrop}.

\medskip
\noindent\emph{Case \textup{(2)}.}
Let $j$ be the integer supplied by option \textup{(2)}. We assume that $\alpha > e^{- c_h F(X)}$, which is equivalent to
\begin{equation}\label{eq:L<F}
L<c_hF(X).
\end{equation}
Option \textup{(2)} gives $y\ll_{h,\epsilon} L^{1+\epsilon}j$ and $R\ge (1+d_2)^j$. Using \eqref{eq:Fdrop} and \eqref{eq:L<F} we obtain
\[
  c_0 y(\log X)^{1/(2+2\epsilon)-1}
\ll_{h,\epsilon} c_0L^{1+\epsilon}j
(\log X)^{1/(2+2\epsilon)-1}
\ll_{h,\epsilon} c_0jF(X)^{1+\epsilon}(\log X)^{1/(2+2\epsilon)-1}=o_{h,\epsilon}(j),
\]
as $X\to\infty$, after enlarging $X_{\text{min}}$ if necessary, we have
\[
c_0( F(X)-F(X'))\le j\log(1+d_2)\le \log R,
\]
which proves \eqref{eq:key-prop} in this case. 
\end{proof}

\begin{rem} Once we have certain choices for $c_0$ and $X_{\text{min}}$ which satisfy the above, the same conclusion holds with $c_0$ replaced by any $c_0'\in(0,c_0]$ and $X_{\text{min}}$ replaced by any $X_{\text{min}}'>X_{\text{min}}$. In particular, we may take $c_0$ as small as needed and $X_{\text{min}}$ as large as needed to perform the iteration procedure below.
\end{rem}

We may now prove Theorem \ref{main} assuming Proposition \ref{density-increment-calc}. 

\begin{proof}[Proof of Theorem \ref{main} assuming Proposition \ref{density-increment-calc}]
Fix $0<\mu<1/2$, and choose $\epsilon>0$ so small that $\mu<\frac{1}{2+2\epsilon}.$ Let $A\subseteq [X]$ have $|A|=\alpha X$ and $(A-A)\cap h(\N)\subseteq\{0\}$. Put $L_0:=\log(1/\alpha)$.

Let $c_h,C_h$ be the constants from Proposition~\ref{density-increment-calc}, applied with this choice of $\epsilon$, and let $c_0$ and $X_{\text{min}}$
be as in Lemma~\ref{lem:GS-propagate}. We also assume that $c_0$ is sufficiently small and that
$X_{\text{min}}$ is sufficiently large for the estimates below. In
particular, enlarge $X_{\text{min}}$ if necessary so that
\[
e^{-c_hF(t)}<c_h
\qquad\text{for every }t\ge X_{\text{min}}.
\]

It suffices to show $\alpha\leq e^{-c_0F(X)}$. Assume for the sake of contradiction that
\begin{equation}\label{eq:L0F}
\alpha> e^{-c_0F(X)} \text{ or equivalently } L_0=\log(1/\alpha)\ <\ c_0F(X)
\ =\ c_0(\log X)^{1/(2+2\epsilon)}.
\end{equation}
Also assume that $X>X_{\text{min}}$. 

Set $A_0:=A$, $X_0:=X$, $\alpha_0:=\alpha$, and $\ell_0:=1$. Since $r_1=0$ and $\lambda(1)=1$, we have $h_{\ell_0}=h$, and hence
\[
(A_0-A_0)\cap h_{\ell_0}(\N)\subseteq\{0\}.
\]
The goal is to iteratively apply Proposition~\ref{density-increment-calc}.
Starting from $n=0$, assume inductively that we have $A_n\subseteq [X_n]$ with density $\alpha_n:=|A_n|/X_n$ and
\[
(A_n-A_n)\cap h_{\ell_n}(\N)\subseteq\{0\}.
\]
We stop the iteration at time $n\geq 0$ if we have at least one of the following conditions.
\begin{center}
\begin{itemize}
\item Option \textup{(1)} holds.
\item $X_n<X_{\text{min}}$.
\item $\ell_n>\exp\big((\log X_n)^{(2+\epsilon)/(2+2\epsilon)}\big).$
\item $\alpha_n>2/3$.

\end{itemize}
\end{center}
Suppose none of these stopping conditions holds at step $n$. Then
\[
\alpha_n>e^{-c_hF(X_n)}.
\]
Indeed, this follows from the failure of option \textup{(1)} when
$\alpha_n\le c_h$, while for $\alpha_n>c_h$ it follows from
$X_n\ge X_{\text{min}}$ and our choice of $X_{\text{min}}$. Thus the
density hypothesis in Lemma~\ref{lem:GS-propagate} is satisfied at every
step that is performed.

If $\alpha_n>c_h$, we use the progression $P_n\subseteq [X_n]$ supplied by \textup{($\ast$)}. If $\alpha_n\le c_h$, then Proposition~\ref{density-increment-calc} supplies a progression $P_n\subseteq [X_n]$ as in \textup{(2)}. In either case, write the common difference of $P_n$ as $\lambda(q_{n+1})$ and write
\[
P_n=\{x_n+\lambda(q_{n+1})y:1\le y\le X_{n+1}\}\subseteq [X_n]
\]
for some integer $x_n$, where $X_{n+1}:=|P_n|$. Define
\[
A_{n+1}:=\{y\in [X_{n+1}]:x_n+\lambda(q_{n+1})y\in A_n\},
\qquad
\ell_{n+1}:=\ell_nq_{n+1}.
\]
Then $\alpha_{n+1}:=|A_{n+1}|/X_{n+1}$ is the density of $A_n$ on $P_n$, and Proposition \ref{inheritance} gives
\[
(A_{n+1}-A_{n+1})\cap h_{\ell_{n+1}}(\N)\subseteq\{0\}.
\]

If the iteration stops in option \textup{(1)} at some index $m$, then, by choosing $c_0$ smaller than $c_h$, we have $\alpha_m\le e^{-c_hF(X_m)}\le e^{-c_0F(X_m)}$. 
Applying Lemma~\ref{lem:GS-propagate} repeatedly back through the previous steps yields
$\alpha=\alpha_0\le e^{-c_0F(X)}$, which contradicts \eqref{eq:L0F}. 

On the other hand, suppose the iteration stops when $X_{m}<X_{\text{min}}$ while $\alpha_m\le 2/3$. We may choose $c_0$ small enough that
\[
\frac23\ \le\ e^{-c_0F(t)} \qquad\text{for all } 1\le t\le X_{\text{min}}.
\]
Consequently, $\alpha_m\le 2/3\le e^{-c_0F(X_m)}$, so, as before, applying Lemma~\ref{lem:GS-propagate} repeatedly back through the previous steps yields
$\alpha=\alpha_0\le e^{-c_0F(X)}$, contradicting \eqref{eq:L0F}. 

Thus, if the iteration terminates in either of the first two stopping conditions, then we are done. Therefore, it remains to show that the iteration must terminate in finitely many steps and to rule out the cases where the iteration stops because of one of the following two conditions:
\begin{itemize}
\item $\ell_m>\exp\big((\log X_m)^{(2+\epsilon)/(2+2\epsilon)}\big).$
\item $\alpha_m>2/3$; 

\end{itemize}

We will prove that the total number of steps is $\ll_{h,\epsilon}L_0$ and further prove that  $\ell_m\leq\exp\big((\log X_m)^{(2+\epsilon)/(2+2\epsilon)}\big).$ Suppose there have been $m$ steps so far in the procedure. Write $\mathcal N_\ast,\mathcal N_2$ for the sets of indices $n\in\{0,\dots,m-1\}$ for which the step
taken is \textup{($\ast$)}, \textup{(2)}, respectively.  If $n\in\mathcal N_2$, let $j_n$ denote the integer supplied by option \textup{(2)}.

\begin{itemize}

\item\emph{Option \textup{($\ast$)}.}
When \textup{($\ast$)} is used we have $\alpha_n>c_h$ and $\alpha_{n+1}\ge \alpha_n+\alpha_n^{C_h}\ge \alpha_n+c_h^{C_h}$.
Since $\alpha_n\le 1$ always, this implies $|\mathcal N_\ast|=O_{h,\epsilon}(1)$.

\item \emph{Option \textup{(2)}.}
Each such step multiplies density by at least $(1+d_2)^{j_n}$, so using $\alpha_n\le 1$ we have
\[
1\ \ge\ \alpha_m\ \ge\ \alpha_0(1+d_2)^{\sum_{n\in\mathcal N_2}j_n}.
\]
Taking logarithms gives
\begin{equation}\label{eq:sum-j}
\sum_{n\in\mathcal N_2}j_n
\le \frac{\log(1/\alpha_0)}{\log(1+d_2)}
=\frac{L_0}{\log(1+d_2)}
\ll_{h,\epsilon} L_0.
\end{equation}

\end{itemize}
Combining the above bounds, we obtain
\begin{equation}\label{eq:mcount}
m\ =\ |\mathcal N_\ast|+|\mathcal N_2|
 \ll_{h,\epsilon}\ L_0.
\end{equation}
Now that we know there are finitely many steps in the procedure, we may assume $m$ is the stopping time. We will prove that, at the stopping time $m$, one has
\begin{equation}\label{ellbound}
\ell_m\le \exp\big((\log X_m)^{(2+\epsilon)/(2+2\epsilon)}\big).
\end{equation}
To this end, we first bound $X_0/X_m$, and to do this we bound the shrinkage factor $X_n/X_{n+1}$ for each increment option, using only the corresponding length bound from
Proposition~\ref{density-increment-calc} and the monotonicity $L_n\le L_0$, where $L_n:=\log(1/\alpha_n)$.

\smallskip
\begin{itemize}
\item \emph{Option \textup{($\ast$)}}. If $n\in\mathcal N_\ast$ then $X_{n+1}\ge \alpha_n^{C_h}X_n$ and $\alpha_n>c_h$, so $X_n/X_{n+1}\le c_h^{-C_h}$.

\item \emph{Option \textup{(2)}}. If $n\in\mathcal N_2$, then option \textup{(2)} gives $X_n/X_{n+1}
\ \le\ \exp\bigl(D_1L_0^{1+\epsilon}j_n\bigr).$

\end{itemize}

\smallskip
Multiplying these bounds over $n=0,\dots,m-1$ and using \eqref{eq:sum-j} yields
\begin{equation}\label{eq:Xratio}
\frac{X_0}{X_m}
\ \ll_{h, \epsilon}\ \exp(D_1L_0^{1+\epsilon}\sum_{n\in \mathcal{N}_2}j_n)\leq \exp(CL_0^{1+\epsilon}L_0),
\end{equation}
where $C$ is some constant depending only on $h$ and $\epsilon$.
In particular,
\[\label{eq:Xm-lower}
X_m\ \ge\ X_0e^{-CL_0^{2+\epsilon}}.
\]
By \eqref{eq:L0F}, $L_0<c_0(\log{X_0})^{1/(2+2\epsilon)}$, so
\begin{equation}
X_m\geq X_0 \exp\left(-Cc_0(\log X_0)^{(2+\epsilon)/(2+2\epsilon)}\right)\geq X_0^{1/2}.
\end{equation}

 Since $X_j\ge X_m>2$ for all $0\le j\le m$, at step $n$ we have $\lambda(q_{n+1})(X_{n+1}-1)<X_n.$
Since $\lambda(q)\ge q$, this gives
\[
q_{n+1}\ \le\ \frac{X_n}{X_{n+1}-1}\ \le\ 2\frac{X_n}{X_{n+1}}.
\]
Multiplying gives
\[\label{eq:ell-upper}
\ell_m=q_1q_2\cdots q_m\ \le\ 2^m\frac{X_0}{X_m}.
\]
By \eqref{eq:mcount}, $m\ll_{h,\epsilon} L_0$, and using \eqref{eq:Xratio}, we obtain
\begin{align*}
\ell_m\ &\le \ \exp\left(C_{h,\epsilon}L_0+CL_0^{2+\epsilon}\right)\\
&\leq \exp(2CL_0^{2+\epsilon})\\&\leq \exp\left(2Cc_0^{2+\epsilon}(\log X_0)^{(2+\epsilon)/(2+2\epsilon)}\right)\\
&\leq \exp\left(\left(\frac{1}{2}\log X_0\right)^{(2+\epsilon)/(2+2\epsilon)}\right),
\end{align*}
where the last inequality follows from choosing $c_0$ small enough.
Since $X_m\geq X_0^{1/2}$, we have $\log{X_m}\geq \frac{1}{2}\log{X_0}$, which gives
\begin{equation}\label{eq:l_m bounded}
    \ell_m\ \le\ \exp\left(\left(\frac{1}{2}\log X_0\right)^{(2+\epsilon)/(2+2\epsilon)}\right)\ \le\ \exp\big((\log X_m)^{(2+\epsilon)/(2+2\epsilon)}\big),
\end{equation}
which proves \eqref{ellbound}.
\medskip

Finally, we show that $\alpha_m\leq 2/3$ throughout the iteration.

Since $A_m$ contains no nonzero difference of the form $h_{\ell_m}(n)$, in particular, it contains no difference $h_{\ell_m}(1)$. Hence the sets
$A_m \text{ and }
A_m+h_{\ell_m}(1)$
are disjoint. Note here that
$A_m\sqcup (A_m+h_{\ell_m}(1))\subseteq [1,X_m +h_{\ell_m}(1)]$ and recall $|A_m|=\alpha_m X_m$. 
This implies
\begin{align*}
 2\alpha_m X_m  & =2|A_m|\leq X_m+h_{\ell_m}(1)\\
 \alpha_m &\leq \frac{1}{2}+\frac{h_{\ell_m}(1)}{2X_m}
    \end{align*}

From Remark \ref{hl bounded} we have  $h_{\ell_m}(1)\ll_h {\ell_m}^{k-1}$ and \eqref{eq:l_m bounded} gives $h_{\ell_m}(1)=o_{h,\epsilon}(X_m)$. This gives

\[
\alpha_m\leq \frac12+o_{h,\epsilon}(1)<\frac23
\]
for $X$ sufficiently large.

\medskip

Thus the iteration cannot stop because of $\alpha_m>2/3$ or because of $\ell_m>\exp\big((\log X_m)^{(2+\epsilon)/(2+2\epsilon)}\big)$. Therefore it must stop either in option \textup{(1)} or with $X_m<X_{\text{min}}$ and $\alpha_m\le 2/3$, and both cases were already shown to imply
$\alpha\le e^{-c_0F(X)}$. This completes the proof of Theorem \ref{main}. 
\end{proof}

\begin{lem}\label{increment}
 Let $X\ge1$, let $A\subseteq[X]$, set $\alpha:=|A|/X$, and define $f_A:=\1_A-\alpha\1_{[X]}$. Let $\eta \in (0,1)$, $q \in\N$, and $\xi \in \R$. Set $T := \max(1, |\xi|X)$. Suppose that either 
\begin{equation}\label{roth-increment} \Big|\widehat{f_A}\Big(\frac{a}{q} + \xi \Big)\Big| \ge \eta \alpha X\end{equation} for some $a \in \Z$ or that
\begin{equation}\label{szem-condition} \sum_{a\mdsub{q}}\Big|\widehat{f_{A}}\Big(\frac{a}{q} + \xi \Big)\Big|^2\ge \eta \alpha^2 X^2, \text{ where } 3 q \leq X .\end{equation}
Then there is a progression $P \subseteq [X]$ with common difference $\lambda(q)$ and length $\gg \lambda(q)^{-1}T^{-1} \eta X$ on which the density of $A$ is at least $(1 + \eta/20) \alpha$.
\end{lem}
\begin{proof}
The proof is nearly identical to the proof of Lemma 6.2 of Green--Sawhney. Since $q\mid\lambda(q)$, each residue class modulo $q$ may be partitioned into progressions of common difference $\lambda(q)$, and the resulting length is $\gg \lambda(q)^{-1}T^{-1}\eta X$.
\end{proof}

\section{Rice's exponential sum estimates}\label{rice-sieve}

In this section we record the sieve definitions and
two exponential sum estimates of Rice \cite{Ric19}. Fix an intersective polynomial $h$ and consider the auxiliary polynomial $h_\ell$ with $\ell\in \N$. For each prime $p$, define $\gamma_\ell(p)$ to be the smallest exponent such that
the derivative $h_\ell'$ is not identically zero modulo $p^{\gamma_\ell(p)}$.
Let $j_\ell(p)$ denote the number of roots of $h_\ell'$ modulo $p^{\gamma_\ell(p)}$. Then $j_\ell(p)\le\min\{k-1,p-1\}$ when $\gamma_\ell(p)=1$ and $j_\ell(p)\le p^{\gamma_\ell(p)-1}(p-1)$ when $\gamma_\ell(p)>1$. For $U\geq 2$ define
\[
W_\ell(U)\ :=\ \Bigl\{ n\in \N : h_\ell'(n)\not\equiv 0 \pmod{p^{\gamma_\ell(p)}} \ \text{ for every prime } p\le U \Bigr\}.
\]
Further, for any $q\in \N$, define
\[
W_\ell^{q}(U)\ :=\ \Bigl\{ n\in \N : h_\ell'(n)\not\equiv 0 \pmod{p^{\gamma_\ell(p)}} \ \text{ for every prime } p\le U
\text{ with } p^{\gamma_\ell(p)}\mid q \Bigr\}.
\]
When $W_\ell^q(U)$ is used to describe a residue class modulo $q$, we use the
following convention. Since every congruence condition in the definition of
$W_\ell^q(U)$ is modulo a prime power $p^{\gamma_\ell(p)}$ dividing $q$,
membership depends only on the residue class modulo $q$. Thus, for
$b\in\Z$, the notation $b\in W_\ell^q(U)$ means that every
positive integer representative of the residue class $b\bmod q$ lies in
$W_\ell^q(U)$.

Next we have a significant circle method input from Alex Rice's paper \cite[Theorem~2.7]{Ric19}. It has two parts. The first is a bound for a complete sum modulo \(q\) taken over selected congruence classes in $W_{\ell}^q(U)$, and the second is a Weyl estimate for exponential sums over $W_{\ell}(U)$. 
\begin{thmB}[Rice]\label{rice-gauss-minor}
Let $\ell\in\N$ and consider the auxiliary polynomial $h_\ell=b_0+b_1n+\cdots+b_kn^k$ of degree $k\geq 2$. Let $\epsilon_0>0$. 
\begin{enumerate}
\item[\textup{(i)}] (\textbf{Square root cancellation})
If $(a,q)=1$, then
\[
\Biggl|\sum_{\substack{0\le s<q\\ s\in W_\ell^{q}(U)}} e\Big(\frac{ah_\ell(s)}{q}\Big)\Biggr|
\ \ll_{h,\epsilon_0}\
\begin{cases}
q^{1/2+\epsilon_0}, & \text{if } q\le U,\\[4pt]
q^{1-1/k+\epsilon_0}, & \text{if }q\geq U.
\end{cases}
\]

\item[\textup{(ii)}] (\textbf{Weyl-type inequality})
If $N,U,Z\ge 2$, $(a,q)=1$ and $|\theta-a/q|<q^{-2}$, then 
\[
\Biggl|\sum_{\substack{1\le n\le N\\ n\in W_\ell(U)}} e\bigl(\theta h_\ell(n)\bigr)\Biggr|
\ \ll_h\
N(\log U)^{ek}
\Biggl(
e^{-\frac{\log Z}{\log U}}
+\left(b_k\log^{k^2}(b_kqN)\Bigl(q^{-1}+\frac{Z}{N}+\frac{qZ^k}{b_kN^k}\Bigr)\right)^{2^{-k}}
\Biggr).
\]
\end{enumerate}
\end{thmB}

\begin{rem}
Here $e=\exp(1)$. Our version of Theorem \ref{rice-gauss-minor} follows immediately from the estimates in \cite{Ric19} using Lucier's content bound $\operatorname{cont}(h_\ell)\ll_h 1$ and the standard divisor bound $\omega(q)\ll \frac{\log q}{\log\log q}$. We also used different letters for the variables so that our $U$ is his $Y$ and our $N$ is his $X$. 
\end{rem}

We now record three useful facts from Rice \cite{Ric19}. The first is a bound for the \emph{content} of an auxiliary polynomial. This is \cite[Lemma~2.6]{Ric19}.
\begin{lem}[Lucier's content bound]
\label{lem:rice-content-bound}
Let $h\in\Z[x]$ be intersective of degree $k\ge 2$, and let $h_\ell$ be the
auxiliary polynomial
\[
h_\ell(n)=\frac{h(r_\ell+\ell n)}{\lambda(\ell)}=b_kn^k+b_{k-1}n^{k-1}+\cdots+b_1n+b_0.
\]
Defining
\[
\operatorname{cont}(h_\ell):=\gcd(b_1,\dots,b_k),
\]
we have
\[
\operatorname{cont}(h_\ell)\ll_h 1
\]
uniformly in $\ell$.
\end{lem}
We say that a polynomial $f\in\mathbb Z[x]$ is identically zero modulo $M$
if $f(n)\equiv 0\pmod M$ for every $n\in\mathbb Z$, or equivalently, if $f/M$ is integer-valued. The following is 
\cite[Proposition~2.5]{Ric19}. 
\begin{lem}
\label{lem:rice-identically-zero-bound}
Let
\[
f(x)=a_0+a_1x+\cdots+a_kx^k\in\Z[x].
\]
If $f$ is identically zero modulo $q$, then
\[
q\mid k!\gcd(a_0,a_1,\dots,a_k).
\]
Consequently, if $f'$ is identically zero modulo $q$, then
\[
q\mid (k-1)!\gcd(a_1,2a_2,\dots,ka_k).
\]
\end{lem}

The following is \cite[equation (13)]{Ric19}. 
\begin{lem}\label{J_USmall}
If $h\in \Z[x]$ is intersective with $\deg h=k$, then for all $\ell\in \N$ and all $U\ge 2$,
\[
\prod_{p\le U}\Bigl(1-\frac{j_\ell(p)}{p^{\gamma_\ell(p)}}\Bigr)
\ \gg_h\
\prod_{k\le p\le U}\Bigl(1-\frac{k-1}{p}\Bigr)
\ \gg_k\ (\log U)^{1-k}.
\]
\end{lem}

\section{Complete moduli and the sieve estimate}\label{completemoduli}
We now introduce the notion of a complete modulus. We say that an integer $q$ is $(\ell,U)$-complete if
\[
p\le U,\ p\mid q
\qquad\Longrightarrow\qquad
p^{\gamma_\ell(p)}\mid q.
\]
Equivalently, for every prime $p\le U$, either $p\nmid q$ or
$p^{\gamma_\ell(p)}\mid q$. For $q\in\N$, let $q^\sharp_{\ell,U}$
denote the smallest $(\ell,U)$-complete multiple of $q$. Explicitly,
\[
q^\sharp_{\ell,U}
=
q\prod_{\substack{p\le U\\0<v_p(q)<\gamma_\ell(p)}}
p^{\gamma_\ell(p)-v_p(q)}.
\]

We prove the following, which says that $q$ and $q^\sharp_{\ell,U}$ are comparable, and in particular $q\asymp_h q^\sharp_{\ell,U}$. 
\begin{lem}\label{complete-denominator-bound}
For every $q\in\N$, one has $q\asymp_h q^\sharp_{\ell,U}$. 
\end{lem}

\begin{proof}
 $q\mid q^\sharp_{\ell,U}$  from the definition of $q^\sharp_{\ell,U}$, hence it is immediate that $q\leq q^\sharp_{\ell,U}$.
It remains to prove that $q^\sharp_{\ell,U}/q$ is bounded in terms of $h$.

Write
\[
h_\ell(n)=b_kn^k+b_{k-1}n^{k-1}+\cdots+b_1n+b_0.
\]
If $\gamma_\ell(p)>1$, then by
definition $h_\ell'$ is identically zero modulo $p^{\gamma_\ell(p)-1}$ in the
sense that
\[
h_\ell'(n)\equiv 0\pmod {p^{\gamma_\ell(p)-1}}
\qquad\text{for every }n\in\mathbb Z.
\]
By Lemma \ref{lem:rice-identically-zero-bound} applied to the polynomial $h_\ell'$ of
degree at most $k-1$, this implies 
\[p^{\gamma_\ell(p)-1}\mid (k-1)!\gcd(b_1,2b_2,\dots,kb_k)\qquad  \text{ for every prime $p$ with }\gamma_{\ell}(p)>1.\]
Therefore
\[
\prod_{\gamma_\ell(p)>1}p^{\gamma_\ell(p)-1}
\le (k-1)!\gcd(b_1,2b_2,\dots,kb_k).
\]
On the other hand, we have 
\[
\gcd(b_1,2b_2,\dots,kb_k)
\le k!\gcd(b_1,\dots,b_k)
= k!\operatorname{cont}(h_\ell).
\]
By the standard content bound for the auxiliary polynomials (Lemma \ref{lem:rice-content-bound}),
we have $\operatorname{cont}(h_\ell)\ll_h1.$
Hence
\begin{equation}\label{completersmall}
\prod_{\gamma_\ell(p)>1}p^{\gamma_\ell(p)-1}\ll_h1.
\end{equation}
Finally,
\[
\frac{q^\sharp_{\ell,U}}{q}
=
\prod_{\substack{p\le U\\0<v_p(q)<\gamma_\ell(p)}}
p^{\gamma_\ell(p)-v_p(q)}
\le
\prod_{\gamma_\ell(p)>1}p^{\gamma_\ell(p)-1}
\ll_h1.
\]
Thus $q^\sharp_{\ell,U}\ll_h q$.
\end{proof}

The following is the sieve estimate over short intervals. We prove this in the appendix by applying the fundamental lemma of sieve theory.
\begin{prop}\label{ricebrunsum2}
There are constants $c,c'>0$, depending only on $h$, such that the following holds. Let $q,\ell\in\N$, let $b\in\Z$, let $U\ge3$, and let $v\ge1$ and $H>0$. Suppose that $q$ is $(\ell,U)$-complete and that
\begin{equation}\label{eq:brun-short-range}
\log(H/q)\ge c\log U\log\log U.
\end{equation}
If $b\notin W_\ell^q(U)$, then
\[
\sum_{\substack{v<n\le t\\n\in W_\ell(U)\\n\equiv b\bmod q}}h_\ell'(n)=0
\]
for every $v\le t\le v+H$. If $b\in W_\ell^q(U)$, then, uniformly for $v\le t\le v+H$,
\begin{equation}\label{eq:ricebrun-short-derivative}
\sum_{\substack{v<n\le t\\n\in W_\ell(U)\\n\equiv b\bmod q}}h_\ell'(n) 
=
\frac{h_\ell(t)-h_\ell(v)}{q}
\prod_{\substack{p\le U\\p^{\gamma_\ell(p)}\nmid q}}
\left(1-\frac{j_\ell(p)}{p^{\gamma_\ell(p)}}\right)
+O_h\left(
\frac{H}{q}h_\ell'(v+H)
e^{-c'\frac{\log(H/q)}{\log U}}
\right).
\end{equation}
In particular, if $v+H\le Y$ and $h_\ell(Y)=X$, then the error term in \eqref{eq:ricebrun-short-derivative} is
\[
O_h\left(
\frac{XH}{qY}
e^{-c'\frac{\log(H/q)}{\log U}}
\right).
\]
\end{prop}
\begin{rem}
The completeness hypothesis in Proposition \ref{ricebrunsum2} cannot simply be omitted. Indeed, take $h_\ell(n)=n^2$ and $q=2$.
Then $h_\ell'(n)=2n$, so
\[
\gamma_\ell(2)=2,\qquad j_\ell(2)=2,
\]
while for every odd prime $p$ one has
\[
\gamma_\ell(p)=1,\qquad j_\ell(p)=1.
\]
Thus, for every $U\ge 3$, the modulus $q=2$ is not $(\ell,U)$-complete, since
$2\mid q$ but $4\nmid q$. Moreover, in this special case,
\[
W_\ell(U)
=
\{n\in\N:p\nmid n\text{ for every prime }p\le U\},
\]
whereas $W_\ell^2(U)=\N,$
because there is no prime $p\le U$ for which $p^{\gamma_\ell(p)}\mid 2$.
Taking $b=2$ and $v=1$, we therefore have $b\in W_\ell^q(U)$, but
\[
\sum_{\substack{1<n\le t\\ n\in W_\ell(U)\\ n\equiv 2\bmod 2}}2n=0
\]
for every $t\ge1$, since every integer congruent to $2\pmod 2$ is even and thus not in $W_\ell(U)$.\\
On the other hand, if we use Proposition \ref{ricebrunsum2} to estimate this sum, the main term in the estimate would be
\[
\frac{t^2-1}{2}
\left(1-\frac{2}{4}\right)
\prod_{\substack{3\le p\le U\\ p\ \mathrm{prime}}}
\left(1-\frac1p\right)
=
\frac{t^2-1}{4}
\prod_{\substack{3\le p\le U\\ p\ \mathrm{prime}}}
\left(1-\frac1p\right),
\]
which is positive for $t>1$. Taking $t=Y$ sufficiently large relative to
$U$ makes the error term smaller than this main term, so the asserted
asymptotic does not necessarily hold without the completeness assumption. 
\end{rem}

When applying the complete sum bound while performing the major arc estimate, we  ``complete $q$" by multiplying both the numerator and denominator of $a/q$ by $q^\sharp_{\ell,U}/q$, and in this way we are able to split the sum into residue classes modulo $q^\sharp_{\ell,U}$. This modulus is complete, so we can apply Proposition \ref{ricebrunsum2}. For this reason, we require the following complete sum estimate. 
\begin{lem}[Complete sums over $(\ell,U)$-complete moduli]\label{bounded-completion-gauss}
Let $\epsilon_0>0$. Let $q\in\N$, let $(a,q)=1$, and let $Q=q^\sharp_{\ell,U}$.
Put $A:=a\frac{Q}{q}.$
Then
\begin{equation}\label{eq:complete-sum-target}
   S:= \left|
\sum_{\substack{s\bmod Q\\s\in W_\ell^{Q}(U)}}
e\left(\frac{Ah_\ell(s)}{Q}\right)
\right|
\ll_{h,\epsilon_0}
\begin{cases}
Q^{1/2+\epsilon_0}, & Q\le U,\\[4pt]
Q^{1-1/k+\epsilon_0}, & Q > U.
\end{cases}
\end{equation}

\end{lem}

\begin{proof}
Let
\[
Q':=\prod_{\substack{p\\v_p(Q)>v_p(q)}}p^{v_p(Q)},
\qquad
Q'':=Q/Q'=\prod_{\substack{p\\v_p(Q)=v_p(q)}}p^{v_p(Q)}.
\]
This choice gives $(Q',Q'')=1$ and $Q=Q'Q''$.

If $p\mid Q'$, then by the definition of $Q=q^\sharp_{\ell,U}$ we have
$0<v_p(q)<\gamma_\ell(p)$ and $v_p(Q)=\gamma_\ell(p)$. In particular
$\gamma_\ell(p)>1$. Hence, using the proof of
Lemma~\ref{complete-denominator-bound},
\begin{equation}\label{eq:Qprime-bounded}
Q'
\le
\prod_{\gamma_\ell(p)>1}p^{\gamma_\ell(p)}
\le
\left(\prod_{\gamma_\ell(p)>1}p^{\gamma_\ell(p)-1}\right)^2
\ll_h 1.
\end{equation}
Also note that $A=aQ/q$ is coprime to $Q''$. 

By the Chinese Remainder Theorem, each residue class $s\bmod Q$ is uniquely
represented by a pair
\[
(s'\bmod Q',\ s''\bmod Q'').
\]
We now translate the condition $s\in W_\ell^Q(U)$ under this decomposition.
By definition, $s\in W_\ell^Q(U)$ means that for every prime $p\le U$ such that $p^{\gamma_\ell(p)}\mid Q$, we have the requirement
\[
h_\ell'(s)\not\equiv0\pmod {p^{\gamma_\ell(p)}}.
\]
Since $(Q',Q'')=1$, we know that if $p\leq U$ and $p^{\gamma_\ell(p)}\mid Q=Q'Q''$, then $p^{\gamma_\ell(p)}\mid Q'$ or $p^{\gamma_\ell(p)}\mid Q''$. Considering these two cases separately, we see that

\begin{equation}\label{eq:WQ-CRT-equivalence}
s\in W_\ell^Q(U)\iff s'\in W_\ell^{Q'}(U)\quad \text{and}\quad s''\in W_\ell^{Q''}(U).
\end{equation}

Choose integers $\overline{Q'}$ and $\overline{Q''}$ such that
\[
Q'\overline{Q'}\equiv1\pmod {Q''},
\qquad
Q''\overline{Q''}\equiv1\pmod {Q'}.
\]
For every integer $x$,
\begin{equation}\label{eq:CRT-exponential-factor}
e\left(\frac{x}{Q}\right)
=
e\left(\frac{\overline{Q''}x}{Q'}\right)
e\left(\frac{\overline{Q'}x}{Q''}\right).
\end{equation}
Applying \eqref{eq:CRT-exponential-factor} with $x=Ah_\ell(s)$, and using
$s\equiv s'\pmod {Q'}$ and $s\equiv s''\pmod {Q''}$, gives
\[
e\left(\frac{Ah_\ell(s)}{Q}\right)
=
e\left(\frac{\overline{Q''}Ah_\ell(s')}{Q'}\right)
e\left(\frac{\overline{Q'}Ah_\ell(s'')}{Q''}\right).
\]
 Using
\eqref{eq:WQ-CRT-equivalence} and the triangle inequality, we obtain
\begin{align}\label{eq:S-CRT-decomposition}
    \left|
\sum_{\substack{s\bmod Q\\s\in W_\ell^{Q}(U)}}
e\left(\frac{Ah_\ell(s)}{Q}\right)
\right|&=\left|\sum_{\substack{s'\bmod Q'\\s'\in W_\ell^{Q'}(U)}} e\left(\frac{\overline{Q''}Ah_\ell(s')}{Q'}\right)\sum_{\substack{s''\bmod Q''\\s''\in W_\ell^{Q''}(U)}}
e\left(\frac{\overline{Q'}Ah_\ell(s'')}{Q''}\right)
\right|\notag\\
&\leq\sum_{\substack{s'\bmod Q'\\ s'\in W_\ell^{Q'}(U)}}
\left|
\sum_{\substack{s''\bmod Q''\\ s''\in W_\ell^{Q''}(U)}}
e\left(\frac{\overline{Q'}Ah_\ell(s'')}{Q''}\right)
\right|.
\end{align}

For fixed $s'\pmod{Q'}$, put
$A'':=A\overline{Q'}.$ Then $(A'',Q'')=1$ because $(A,Q'')=1$ and
$(\overline{Q'},Q'')=1$. 
Applying Theorem~\ref{rice-gauss-minor}\,\textup{(i)} to each inner sum in
\eqref{eq:S-CRT-decomposition} and using \eqref{eq:Qprime-bounded} gives
\begin{equation*}\label{eq:S-Qdoubleprime-bound}
S
\ll_{h,\epsilon_0}
\begin{cases}
Q''^{1/2+\epsilon_0}, & \text{if }Q''\le U,\\[4pt]
Q''^{1-1/k+\epsilon_0}, & \text{if }Q''\ge U.
\end{cases}
\end{equation*}
Since $k\ge2$,
\[
S\ll_{h,\epsilon_0}Q''^{1-1/k+\epsilon_0}\le Q^{1-1/k+\epsilon_0}
\]
regardless of the size of $Q''$ relative to $U$. Moreover, if $Q\le U$, then $Q''\le Q\le U$, so
\[
S\ll_{h,\epsilon_0} Q''^{1/2+\epsilon_0}\le Q^{1/2+\epsilon_0}.
\]
Combining these two estimates proves \eqref{eq:complete-sum-target}, and hence the lemma.

\end{proof}

\section{The Weyl estimate}\label{weylestimate}
Recall that $h$ and consequently $h_\ell$ are positive, increasing, and concave up on $[1,\infty)$. Also, define 
\[
I(h_\ell):=\{h_\ell(n): n\in W_\ell(U)\},
\]
where $U>0$ is a parameter to be chosen later. We define a smooth weight following \cite{GreenSawhney2024FSv1}. We begin by recording the following result, which can be obtained by applying simple horizontal transformations to the function found in Lemma $5.1$ of \cite{GreenSawhney2024FSv1}.
\begin{lem}
There exists a smooth function $w:\R \to [0,1]$, supported on $[0,1]$, such that $w(x) \ge \frac{1}{2}$ for $x \in [\frac{1}{4}, \frac{3}{4}]$ and satisfying the Fourier estimate $|\widehat{w}(t)|\ll e^{-\sqrt{|t|/2}}$ for all $t \in \R$.
\end{lem}

Let $w$ be given by the lemma above. For the auxiliary polynomial $h_\ell$, define $g_{X,h_\ell}: \Z\to \R$ by
\[
g_{X,h_\ell}(x):=
\begin{cases}
0, & \text{if } |x|\notin I(h_\ell),\\
J_{U,\ell}\cdot h_\ell '(n) w\left(\frac{h_\ell (n)}{X}\right), & \text{if } |x|=h_\ell(n)\in I(h_\ell),
\end{cases}
\]
where 
\[
J_{U,\ell}:=\prod_{p\le U}
\left(1-\frac{j_\ell(p)}{p^{\gamma_\ell(p)}}\right)^{-1}>0.
\]
We know that $h_\ell(n)$ is positive and increasing for $n\in \N$ by \eqref{positivity2}, so $g_{X,h_\ell}$ is well-defined. Notice also that $g_{X,h_\ell}$ is even and supported on $\pm I(h_\ell)$. 

Recall that $\epsilon>0$ is a fixed small parameter. Fix $\gamma:=1/2$ and
let $X$ be a large ambient parameter and make the global choices
\begin{equation}\label{eq:UZ-choice} K=2^k\qquad
U:=\exp\big((\log X)^\gamma\big),
\qquad
Z:=\exp\big((\log X)^{\gamma+1/(2+2\epsilon)}\big).
\end{equation}
In the estimates below we assume
\begin{equation}\label{ell-quasi-bound}
\ell\le \exp\big((\log X)^{(2+\epsilon)/(2+2\epsilon)}\big).
\end{equation}
Since $h_\ell(1)\ll_h\ell^{k-1}=X^{o_{h,\epsilon}(1)}$ for $X$ sufficiently large, \eqref{positivity2} gives a unique $Y>1$ such that $h_\ell(Y)=X$.
In the application we will apply Theorem \ref{rice-gauss-minor}\,(ii) for integers $N$ satisfying
\begin{equation}\label{N-range}
X^{1/k-r_0} \ll N \ll X^{1/k},
\end{equation}
where $0<r_0<1/k$ is fixed. We have $\log N\asymp\log X$ throughout the range \eqref{N-range}.
Consequently, any quasipolynomial factor of the form $\exp((\log X)^d)$ for $d<1$ is also $N^{o(1)}$ on this range.
In particular, with $U,Z$ as in \eqref{eq:UZ-choice} we have
\begin{equation}\label{minorconditions}
Z^k=\exp\big(k(\log X)^{\gamma+1/(2+2\epsilon)}\big)=N^{o(1)},
\quad
e^{-\frac{\log Z}{\log U}}=\exp\big(-(\log X)^{1/(2+2\epsilon)}\big)=N^{-o(1)}.
\end{equation}

\begin{lem}\label{Weyl}
Let $X$ be sufficiently large, let $U,Z$ be as in \eqref{eq:UZ-choice}, and let $N$ satisfy \eqref{N-range}.
Let $\theta\in\R/\Z$ and suppose
\begin{equation}\label{eq:large-corr}
\Biggl|\sum_{\substack{1\le n\le N\\ n\in W_\ell(U)}} e\bigl(h_\ell(n)\theta\bigr)\Biggr|\ \ge\ \delta N.
\end{equation}
Assume also that
\begin{equation}\label{eq:delta-threshold}
\delta\ \ge\ 2 E_h(\log X)^{ek}
\Bigl(
e^{-(\log X)^{1/(2+2\epsilon)}}+\Bigl(\frac{b_kZ}{N}\log^{k^2}(b_kN^{k+1})\Bigr)^{1/K}
\Bigr),
\end{equation}
where $E_h>0$ is the implicit constant in Theorem~\ref{rice-gauss-minor}\,(ii). 

Then there exists $q\in\N$ such that
\[
q\ \ll_{h}T^{K}\delta^{-K}
\qquad\text{and}\qquad
\|q\theta\|_{\mathbb T}\ \ll_{h}T^{K}N^{-k+o_{h,\epsilon}(1)}\delta^{-K},
\]
where 
\[
T:=(\log X)^{ek}\left(\log^{k^2}(b_kN^{k+1})b_k\right)^{1/K}.
\]
\end{lem}

\begin{proof}
Write
\[
S(\theta)\ :=\ \sum_{\substack{1\le n\le N\\ n\in W_\ell(U)}} e\bigl(h_\ell(n)\theta\bigr),
\]
so that $|S(\theta)|\ge \delta N$ by \eqref{eq:large-corr}. 
Define 
\begin{equation}\label{eq:Q-def}
Q\ :=\ \left(\frac{\delta}{4E_hT}\right)^{K}\cdot\frac{N^{k}}{Z^{k}}.
\end{equation}
The second term in \eqref{eq:delta-threshold} gives $Q\gg (N/Z)^{k-1}\geq1$, for $X$ sufficiently large.

Choose the representative of $\theta$ in $[-1/2,1/2)$, which we continue to denote by $\theta$. By Dirichlet's theorem there exist coprime integers $a,q$ with $1\le q\le Q$ such that
\[
\Bigl|\theta-\frac{a}{q}\Bigr|\ \le\ \frac{1}{qQ}\ \le\ \frac{1}{q^{2}}.
\]
By Theorem~\ref{rice-gauss-minor}\,(ii), using $(x+y+z)^{1/K}\le x^{1/K}+y^{1/K}+z^{1/K}$ for $x,y,z\ge 0$ and dividing by $N$, we obtain
\begin{align}
\delta \le\ \frac{|S(\theta)|}{N}
\leq E_h\left(
(\log X)^{ek}e^{-(\log X)^{1/(2+2\epsilon)}}
+Tq^{-1/K}
+T\Bigl(\frac{Z}{N}\Bigr)^{1/K}
+Tq^{1/K}\Bigl(\frac{Z^k}{b_kN^k}\Bigr)^{1/K}\right).
\label{eq:weyl-raw}
\end{align}
 For the rightmost term, using \eqref{eq:Q-def} we have 
\[
E_hq^{1/K}T\Bigl(\frac{Z^k}{b_kN^k}\Bigr)^{1/K}
\le E_hQ^{1/K}T\Bigl(\frac{Z^k}{N^k}\Bigr)^{1/K}
= \frac{\delta}{4}.
\]
By the assumption \eqref{eq:delta-threshold} in the lemma, we have 
\[
E_h\left((\log X)^{ek}e^{-(\log X)^{1/(2+2\epsilon)}}
+T\Bigl(\frac{Z}{N}\Bigr)^{1/K}\right)\leq \delta/2.
\]
Combining these facts, we see that $\frac{\delta}{4}\ \ll_h \ Tq^{-1/K}, $ and hence $q\ \ll_h\ T^{K}\delta^{-K}.$ Finally, from $\|q\theta\|_{\mathbb T}\le 1/Q$ and using the definition of $Q$ as in \eqref{eq:Q-def} together with $Z^k=N^{o(1)}$ we have
\[
\|q\theta\|_{\mathbb T}\ll_{h}T^KN^{-k+o_{h,\epsilon}(1)}\delta^{-K},
\qedhere\]
\end{proof}
\medskip

\begin{lem}\label{Mukul}
Let $X$ be large. Assume \eqref{ell-quasi-bound}. There exists a constant $D_h>0$ such that the following holds. Let $0<\delta<1$ satisfy
\[
\frac{\delta}{J_{U,\ell}}\geq D_h (\log X)^{ek}\exp\!\Big(-(\log X)^{1/(2+2\epsilon)}\Big).
\]
If $\theta\in\R/\Z$ and $|\widehat{g_{X,h_\ell}}(\theta)|\ge \delta X$, then there exists $q\in\N$ such that
\[
q\ll_{h} \delta^{-K}X^{o_{h,\epsilon}(1)},
\qquad
\|q\theta\|_{\mathbb T}\ll_{h}\delta^{-K-k}X^{-1+o_{h,\epsilon}(1)}.
\]
\end{lem}

\begin{proof}
Let $C$ be a positive constant depending on $w$ but independent of all other parameters.
Let $Y$ be such that $h_\ell(Y)=X$. Since $w$ is supported on $[0,1]$, the definition of $g_{X,h_\ell}$ gives
\begin{equation}\label{eq:ghat}
    \widehat{g_{X,h_\ell}}(\theta)
=2J_{U,\ell}\cdot\Re\sum_{\substack{1\le n\le Y\\ n\in W_\ell(U)}} h_\ell'(n)
w\Big(\frac{h_\ell(n)}{X}\Big)e(-\theta h_\ell(n)).
\end{equation}

From $|\widehat{g_{X,h_\ell}}(\theta)|\ge \delta X$ and $|\Re z|\le |z|$, we obtain
\[
\Bigg|\sum_{\substack{1\le n\le Y\\ n\in W_\ell(U)}} h_\ell'(n)
w\Big(\frac{h_\ell(n)}{X}\Big)e(-\theta h_\ell(n))\Bigg|
\ \ge\ \frac{\delta X}{2J_{U,\ell}}.
\]
Expanding $w$ using the inverse Fourier transform and using the triangle inequality gives
\begin{equation}\label{eq:delta/Jul}
    \int_{-\infty}^{\infty} |\widehat{w}(\xi)|
\Bigg|\sum_{\substack{1\le n\le Y\\ n\in W_\ell(U)}} h_\ell'(n)
e\Big(\Big(-\theta+\frac{\xi}{X}\Big)h_\ell(n)\Big)\Bigg|\,d\xi
\ \gg_h\ \frac{\delta}{J_{U,\ell}}X.
\end{equation}

Since $h_{\ell}''>0$ on $[1,\infty)$, using Markov brothers' inequality (see e.g. equation (1.1) in \cite{EKS00}) we obtain
\begin{equation}\label{eq:hl' bound}
    \max_{1\le n\le Y} h_\ell'(n)\leq h_\ell'(Y)\ll_{h}\frac{X}{Y}. 
\end{equation}

Now, we trivially bound the inner sum in \eqref{eq:delta/Jul} using the triangle inequality and \eqref{eq:hl' bound}:
\begin{equation}\label{eq:Mukul:trivial}
\Bigg|\sum_{\substack{1\le n\le Y\\ n\in W_\ell(U)}} h_\ell'(n)
e\Big(\Big(-\theta+\frac{\xi}{X}\Big)h_\ell(n)\Big)\Bigg|
\ \le\ \sum_{1\le n\le Y} h_\ell'(n)
\ \ll_h\ Y\max_{1\le n\le Y} h_\ell'(n)
\ \ll_h\ X.
\end{equation}
Using this and the Fourier decay $|\widehat w(\xi)|\ll \exp(-\sqrt{|\xi|/2})$, the contribution of $|\xi|\ge C/(\delta/J_{U,\ell})$
to the outer integral is $o((\delta/J_{U,\ell})X)$ provided $C$ is chosen large enough.
Thus there exists some $\xi$ with $|\xi|\le C/(\delta/J_{U,\ell})$ such that
\begin{equation}\label{eq:Mukul:weightedcorr}
\Bigg|\sum_{\substack{1\le n\le Y\\ n\in W_\ell(U)}} h_\ell'(n)
e\Big(\Big(-\theta+\frac{\xi}{X}\Big)h_\ell(n)\Big)\Bigg|
\ \gg_h\ \frac{\delta}{J_{U,\ell}}X.
\end{equation}

Write $\beta:=-\theta+\xi/X$. Changing the order of summation and letting $t(u)$ denote the inverse function of $h_\ell'$ on
$[h_\ell'(1),h_\ell'(Y)]$, and setting $t(u)=1$ for $0\le u<h_\ell'(1)$, we have
\[
\sum_{\substack{1\le n\le Y\\ n\in W_\ell(U)}} h_\ell'(n)e(\beta h_\ell(n))
=\sum_{\substack{1\le n\le Y\\ n\in W_\ell(U)}} \int_{0}^{h_\ell'(n)} e(\beta h_\ell(n))\, du
=\int_{0}^{h_\ell'(Y)} \sum_{\substack{t(u)\le n\le Y\\ n\in W_\ell(U)}} e(\beta h_\ell(n))\, du.
\]
Combining this with \eqref{eq:Mukul:weightedcorr} and substituting $u=yh_\ell'(Y)$ gives
\[
\Bigg|\int_{0}^{1} \sum_{\substack{t(yh_\ell'(Y))\le n\le Y\\ n\in W_\ell(U)}} e(\beta h_\ell(n))\,dy\Bigg|
\ \gg_h\ \frac{(\delta/J_{U,\ell})X}{h_\ell'(Y)}.
\]
By \eqref{eq:hl' bound}, $h_\ell'(Y)\ll_h X/Y$, so we can rewrite this as
\[
\Bigg|\int_{0}^{1} \sum_{\substack{t(yh_\ell'(Y))\le n\le Y\\ n\in W_\ell(U)}} e(\beta h_\ell(n))\,dy\Bigg|
\ \gg_h\ \frac{\delta}{J_{U,\ell}}Y.
\]
Then for some $y\in[0,1]$ we have
\begin{equation}\label{eq:Mukul:tail}
\Bigg|\sum_{\substack{t(yh_\ell'(Y))\le n\le Y\\ n\in W_\ell(U)}} e(\beta h_\ell(n))\Bigg|
\ \gg_h\ \frac{\delta}{J_{U,\ell}}Y.
\end{equation}
Let $M:=\lceil t(yh_\ell'(Y))\rceil$ and write
\[
S(N):=\sum_{\substack{1\le n\le N\\ n\in W_\ell(U)}} e(\beta h_\ell(n)).
\]
Then \eqref{eq:Mukul:tail} is $|S(Y)-S(M-1)|\gg_h \frac{\delta}{J_{U,\ell}}Y$, so for some $N\le Y$ we have
$|S(N)|\gg_h \frac{\delta}{J_{U,\ell}}Y$, so that 
\begin{equation}\label{eq:Mukul:largecorr}
\Bigg|\frac1N\sum_{\substack{1\le n\le N\\ n\in W_\ell(U)}} e(\beta h_\ell(n))\Bigg|
\ \gg_h\ \frac{\delta}{J_{U,\ell}}.
\end{equation}
By the triangle inequality, this also gives $N\gg_h \frac{\delta}{J_{U,\ell}}Y$.

The goal is to apply Lemma~\ref{Weyl} to the phase $\beta$ with a sufficiently small constant multiple of $\delta/J_{U,\ell}$ in place of $\delta$. 

By Lemma~\ref{lem:invert_hl} we have $Y=(X/b_k)^{1/k}+O_h(1)$, and since $b_k\ll_h \ell^{k-1}$, the bound \eqref{ell-quasi-bound} gives $Y=X^{1/k-o_{h,\epsilon}(1)}$. In particular $N\le Y\ll_h X^{1/k}$ and $N\gg_h \frac{\delta}{J_{U,\ell}}Y$,
so $N$ satisfies \eqref{N-range} for $X$ large.

Moreover the hypothesis
\[
\frac{\delta}{J_{U,\ell}}\geq D_h(\log X)^{ek}\exp(-(\log X)^{1/(2+2\epsilon)})
\]
ensures the condition \eqref{eq:delta-threshold}
in Lemma~\ref{Weyl} if $D_h\geq 2E_h$.  This is because the remaining term in \eqref{eq:delta-threshold} is 
\[2E_hT\left(\frac ZN\right)^{1/K}\ll_h(\log{N})^{c}\left(\frac{\exp(O_{h,\epsilon}(\log{N})^{\gamma+1/(2+2\epsilon)})}{N}\right)^{1/K}\ll_h N^{-1/(K+1)}=o_{h,\epsilon}(\delta/J_{U,\ell}),\]
 Here $X$ is sufficiently large and we use
$N\gg_h \frac{\delta}{J_{U,\ell}}Y$, $b_k\ll_h \ell^{k-1}$, $\log N\asymp_{h,\epsilon} \log X$ \eqref{ell-quasi-bound}, and the definition of $Z$.

Thus Lemma~\ref{Weyl} yields an integer $q\ge 1$ such that
\[
q\ll_h T^{K} J_{U,\ell}^{K}\delta^{-K},
\qquad
\|q\beta\|_{\mathbb T}\ll_h T^{K} J_{U,\ell}^{K}N^{-k+o_{h,\epsilon}(1)}\delta^{-K},
\]
where $T$ is as in Lemma~\ref{Weyl}. 
Using the definition of $T$, $\log N\asymp_h \log X$ and Lemma \ref{J_USmall}, we obtain \[q\ll_h b_k(\log X)^{c_4}(\log U)^{(k-1)K}\delta^{-K}\ll_hX^{o_{h,\epsilon}(1)}\delta^{-K},\]
\[\|q\beta\|_{\mathbb T}\ll_h b_k^2(\log X)^{c_4}(\log U)^{(k-1)(K+k)}X^{-1+o_h(1)}\delta^{-(K+k)}\ll_hX^{-1+o_{h,\epsilon}(1)}\delta^{-(K+k)}\]

for $c_4>0$, where we have used $\frac{\delta}{J_{U,\ell}}Y\ll_h N$ and $Y^{-k}\ll_h b_k/X$ by Lemma~\ref{lem:invert_hl}.

Finally, since $\beta=-\theta+\xi/X$ and $|\xi|\le C/(\delta/J_{U,\ell})$, we have
\[
\|q\theta\|_{\mathbb T}\le \|q\beta\|_{\mathbb T}+\frac{|q\xi|}{X}
\ll_h \delta^{-K-k}X^{-1+o_{h,\epsilon}(1)},
\]
which completes the proof.
\end{proof}

\section{The major arc estimate}\label{majorarcestimate}
Next, we estimate $\widehat{g_{X,h_\ell}}(\xi)$ in the case where $\xi$ is close to a rational number $a/q$ with a small denominator. We sum over disjoint short intervals on which the summand has bounded variation. In this way we achieve a more favorable error term, which is key to achieving the desired exponent range in Theorem \ref{main}; see \cite[Theorems 2a--2b]{Vinogradov54}.

\begin{lem}[Major arc estimate]\label{Rishika}
Let $q\in\N$ and $a\in\Z$ be such that $q<X^{1/(8k)}$ and $(a,q)=1$. Assume
\[
\log(Y/q)\ge 2c\log U\log\log U,
\]
where $Y$ is uniquely determined by $h_\ell(Y)=X$. Let $\theta\in\R$ satisfy $|\theta|\le X^{-1+1/(8k)}$, and assume $ \ell\le \exp\big((\log X)^{(2+\epsilon)/(2+2\epsilon)}\big) $. Then
\[
\Bigl|\widehat{g_{X,h_\ell}}\Bigl(\frac{a}{q}+\theta\Bigr)\Bigr|
\ \ll_{h,\epsilon}\
Xe^{-\sqrt{|\theta|X/2}}q^{-1/(k+\epsilon)}
+
J_{U,\ell}Xe^{-c_2\frac{\log(Y/q)}{\log U}},
\]
for $c_2>0.$ Moreover, if $q^\sharp_{\ell,U}\le U$, then the stronger estimate
\[
\Bigl|\widehat{g_{X,h_\ell}}\Bigl(\frac{a}{q}+\theta\Bigr)\Bigr|
\ \ll_{h,\epsilon}\
Xe^{-\sqrt{|\theta|X/2}}q^{-1/(2+\epsilon)}
+
J_{U,\ell}Xe^{-O_{h,\epsilon}\left(\frac{\log(Y/q)}{\log U}\right)}
\]
also holds.
\end{lem}

\begin{proof}
Put
\[
Q:=q^\sharp_{\ell,U},
\qquad
A:=a\frac{Q}{q}.
\]
Then $Q$ is $(\ell,U)$-complete, $Q\ll_hq$ by Lemma~\ref{complete-denominator-bound}, and $a/q=A/Q$.

Set
\[
R:=\bigl\lceil1+|\theta|X\bigr\rceil,
\qquad
H:=\frac{Y-1}{R},
\qquad
v_j:=1+jH
\quad(0\le j\le R).
\]
Thus the intervals $(v_{j-1},v_j]$, with $1\le j\le R$, partition $(1,Y]$. Lemma~\ref{lem:invert_hl} gives
$
\log Y=\left(\frac1k+o_{h,\epsilon}(1)\right)\log X$. Moreover, since $q<X^{1/(8k)}$, we have
\[
\log R\le\frac{1}{8k}\log X+O_k(1),
\qquad
\log(Y/q)\ge\left(\frac{7}{8k}+o_{h,\epsilon}(1)\right)\log X.
\]
Since $Q\ll_hq$, it follows that, for $X$ sufficiently large,
\begin{equation}\label{eq:major-short-sieve-range}
\log(H/Q)
\ge\log(Y/q)-\log R-O_h(1)
\ge\frac12\log(Y/q)
\ge c\log U\log\log U.
\end{equation}

Define
\[
G(t):=w\left(\frac{h_\ell(t)}X\right)e(-\theta h_\ell(t)).
\]
Fix $b\in W_\ell^Q(U)$ and $1\le j\le R$. We first estimate the contribution of the residue class $b\bmod Q$ on the short interval $(v_{j-1},v_j]$, namely
\[
\sum_{\substack{v_{j-1}<n\le v_j\\n\in W_\ell(U)\\n\equiv b\bmod Q}}
h_\ell'(n)G(n).
\]
The partial sums of the coefficients $h_\ell'(n)$ appearing in this expression are estimated by Proposition~\ref{ricebrunsum2}. More precisely, uniformly for $v_{j-1}\le t\le v_j$,
\begin{equation}\label{eq:major-short-sieve}
\sum_{\substack{v_{j-1}<n\le t\\n\in W_\ell(U)\\n\equiv b\bmod Q}}
h_\ell'(n) 
=
\frac{h_\ell(t)-h_\ell(v_{j-1})}{Q}
\prod_{\substack{p\le U\\p^{\gamma_\ell(p)}\nmid Q}}
\left(1-\frac{j_\ell(p)}{p^{\gamma_\ell(p)}}\right)
+E_{b,j}(t),
\end{equation}
where
\begin{equation}\label{eq:Ebj bound}
    |E_{b,j}(t)|
\ll_h
\frac{XH}{QY}
e^{-c'\frac{\log(H/Q)}{\log U}}.
\end{equation}

We now apply partial summation to the weighted sum over $(v_{j-1},v_j]$. Written out in full, partial summation gives
\begin{align*}
&\sum_{\substack{v_{j-1}<n\le v_j\\n\in W_\ell(U)\\n\equiv b\bmod Q}}
h_\ell'(n)G(n)=
G(v_j)
\sum_{\substack{v_{j-1}<n\le v_j\\n\in W_\ell(U)\\n\equiv b\bmod Q}}
h_\ell'(n)
-\int_{v_{j-1}}^{v_j}
\sum_{\substack{v_{j-1}<n\le t\\n\in W_\ell(U)\\n\equiv b\bmod Q}}
h_\ell'(n)G'(t)\,dt.
\end{align*}
There is no lower-endpoint term because the inner sum vanishes when $t=v_{j-1}$. We insert \eqref{eq:major-short-sieve} into this identity. The contribution of the main term in \eqref{eq:major-short-sieve} is
\begin{align*}
&~\quad\frac1Q
\prod_{\substack{p\le U\\p^{\gamma_\ell(p)}\nmid Q}}
\left(1-\frac{j_\ell(p)}{p^{\gamma_\ell(p)}}\right)
\cdot
\left(
\bigl(h_\ell(v_j)-h_\ell(v_{j-1})\bigr)G(v_j)
-\int_{v_{j-1}}^{v_j}
\bigl(h_\ell(t)-h_\ell(v_{j-1})\bigr)G'(t)\,dt
\right)\\
&=
\frac1Q
\prod_{\substack{p\le U\\p^{\gamma_\ell(p)}\nmid Q}}
\left(1-\frac{j_\ell(p)}{p^{\gamma_\ell(p)}}\right)
\int_{v_{j-1}}^{v_j}h_\ell'(t)G(t)\,dt,
\end{align*}
where the final equality follows by integration by parts. The contribution of the error term $E_{b,j}(t)$ is
\[
E_{b,j}(v_j)G(v_j)
-\int_{v_{j-1}}^{v_j}E_{b,j}(t)G'(t)\,dt.
\]

It remains to bound this error contribution. We have $|G(v_j)|\le1$, and using the triangle inequality
\begin{align*}
|G'(t)|
&\leq 
\left|\frac{h_\ell'(t)}Xw'\left(\frac{h_\ell(t)}X\right)e(-\theta h_\ell(t))\right|
+\left|2\pi i\theta h_\ell'(t)w\left(\frac{h_\ell(t)}X\right)e(-\theta h_\ell(t))\right|\\
&\ll\frac{h_{\ell}'(t)}{X}(1+|\theta|X)\\
&\ll R\frac{h_\ell'(t)}{X}.
\end{align*}
Since $h_\ell''>0$ and $h_\ell'(Y)\ll_hX/Y$ using \eqref{eq:hl' bound},
\begin{equation}\label{eq:major-block-change}
\int_{v_{j-1}}^{v_j}|G'(t)|\,dt
\ll
\frac{R}{X}\bigl(h_\ell(v_j)-h_\ell(v_{j-1})\bigr)
\le
\frac{RH}{X}h_\ell'(Y)
\ll_h1.
\end{equation}
It follows from \eqref{eq:Ebj bound} that the absolute value of the contribution of $E_{b,j}(t)$ is at most
\[
\sup_{v_{j-1}\le t\le v_j}|E_{b,j}(t)|
\left(
|G(v_j)|
+\int_{v_{j-1}}^{v_j}|G'(t)|\,dt
\right)
\ll_h
\frac{XH}{QY}
e^{-c'\frac{\log(H/Q)}{\log U}}.
\]
We have therefore shown that
\[
\sum_{\substack{v_{j-1}<n\le v_j\\n\in W_\ell(U)\\n\equiv b\bmod Q}}
h_\ell'(n)G(n)
=
\frac1Q
\prod_{\substack{p\le U\\p^{\gamma_\ell(p)}\nmid Q}}
\left(1-\frac{j_\ell(p)}{p^{\gamma_\ell(p)}}\right)
\int_{v_{j-1}}^{v_j}h_\ell'(t)G(t)\,dt
+O_h\left(
\frac{XH}{QY}
e^{-c'\frac{\log(H/Q)}{\log U}}
\right).
\]
Summing over $1\le j\le R$ and using $RH=Y-1$, we obtain
\begin{equation}\label{eq:major-residue-class}
\sum_{\substack{1<n\le Y\\n\in W_\ell(U)\\n\equiv b\bmod Q}}
h_\ell'(n)G(n) 
=
\frac1Q
\prod_{\substack{p\le U\\p^{\gamma_\ell(p)}\nmid Q}}
\left(1-\frac{j_\ell(p)}{p^{\gamma_\ell(p)}}\right)
\int_1^Yh_\ell'(t)G(t)\,dt
+O_h\left(
\frac XQe^{-c'\frac{\log(H/Q)}{\log U}}
\right).
\end{equation}
The substitution $u=h_\ell(t)/X$ gives
\[
\int_1^Yh_\ell'(t)G(t)\,dt
=X\int_{h_\ell(1)/X}^1w(u)e(-\theta Xu)\,du
=X\widehat w(\theta X)+O_h(\ell^{k-1}),
\]
since $h_\ell(1)\ll_h\ell^{k-1}$. Hence, for every $b\in W_\ell^Q(U)$,
\begin{equation}
\sum_{\substack{1<n\le Y\\n\in W_\ell(U)\\n\equiv b\bmod Q}}
h_\ell'(n)w\left(\frac{h_\ell(n)}X\right)e(-\theta h_\ell(n))=
\frac XQ\widehat w(\theta X)
\prod_{\substack{p\le U\\p^{\gamma_\ell(p)}\nmid Q}}
\left(1-\frac{j_\ell(p)}{p^{\gamma_\ell(p)}}\right)
+O_h\left(\frac{\ell^{k-1}}Q+
\frac XQe^{-c'\frac{\log(H/Q)}{\log U}}
\right).
\label{eq:major-residue-class-final}
\end{equation}
We now write the expression for $\widehat{g_{X,h_\ell}}\left(\frac aq+\theta\right)$ as in \eqref{eq:ghat}.
If $n\equiv b\pmod Q$, then $e\left(-ah_\ell(n)/q\right)
=e\left(-Ah_\ell(b)/Q\right).$
Splitting into residue classes modulo $Q$, applying \eqref{eq:major-residue-class-final}, and bounding the term $n=1$ separately, we obtain
\begin{align}
\widehat{g_{X,h_\ell}}\left(\frac aq+\theta\right)
&=
2J_{U,\ell}\Re\Bigg[
\frac XQ\widehat w(\theta X)
\prod_{\substack{p\le U\\p^{\gamma_\ell(p)}\nmid Q}}
\left(1-\frac{j_\ell(p)}{p^{\gamma_\ell(p)}}\right)\cdot
\sum_{\substack{b\bmod Q\\b\in W_\ell^Q(U)}}
e\left(-\frac{Ah_\ell(b)}Q\right)
\Bigg] \notag\\
&\quad
+O_h\left(J_{U,\ell}\ell^{k-1}+
J_{U,\ell}Xe^{-c'\frac{\log(H/Q)}{\log U}}
\right).
\label{eq:major-completed}
\end{align}
We first bound the error terms. Using $\ell\le \exp\big((\log X)^{(2+\epsilon)/(2+2\epsilon)}\big)$ and Lemma~\ref{J_USmall} for the first term and using \eqref{eq:major-short-sieve-range} for the second term, we get
\[ J_{U,\ell}\ell^{k-1}+J_{U,\ell}Xe^{-c'\frac{\log(H/Q)}{\log U}}\ll X^{o_{h,\epsilon}(1)}+J_{U,\ell}Xe^{-\frac{c'}{2}\frac{\log(Y/q)}{\log U}}\ll J_{U,\ell}X e^{-c_2\frac{\log(Y/q)}{\log U}}.\] 
Thus this term has the required form after decreasing $c_2$ if necessary.

By the definition of $J_{U,\ell}$,
\[
J_{U,\ell}
\prod_{\substack{p\le U\\p^{\gamma_\ell(p)}\nmid Q}}
\left(1-\frac{j_\ell(p)}{p^{\gamma_\ell(p)}}\right)
=
\prod_{\substack{p\le U\\p^{\gamma_\ell(p)}\mid Q}}
\left(1-\frac{j_\ell(p)}{p^{\gamma_\ell(p)}}\right)^{-1}.
\]
The primes with $\gamma_\ell(p)>1$ contribute $O_h(1)$ as in the proof of Lemma~\ref{complete-denominator-bound}, since $j_\ell(p)\le p^{\gamma_\ell(p)-1}(p-1)$. The primes $p\le2k$ also contribute $O_h(1)$. For every remaining prime, $\gamma_\ell(p)=1$, $p>2k$, and $j_\ell(p)\le k-1$, so $\left(1-j_\ell(p)/p\right)^{-1}<2.$
Consequently, for every $\eta>0$,
\[
J_{U,\ell}
\prod_{\substack{p\le U\\p^{\gamma_\ell(p)}\nmid Q}}
\left(1-\frac{j_\ell(p)}{p^{\gamma_\ell(p)}}\right)
\ll_{h,\eta}2^{\omega(Q)}\ll_\eta Q^\eta.
\]

Lemma~\ref{bounded-completion-gauss} gives
\[
\left|
\sum_{\substack{b\bmod Q\\b\in W_\ell^Q(U)}}
e\left(-\frac{Ah_\ell(b)}Q\right)
\right|
\ll_{h,\eta}
\begin{cases}
Q^{1/2+\eta},&\text{if }Q\le U,\\
Q^{1-1/k+\eta},&\text{if }Q\ge U.
\end{cases}
\]
Since $Q\asymp_hq$ and $|\widehat w(\theta X)|\ll e^{-\sqrt{|\theta|X/2}}$, choosing $\eta>0$ sufficiently small in terms of $\epsilon$ and applying \eqref{eq:major-completed} proves the first estimate. If $Q=q^\sharp_{\ell,U}\le U$, the square-root estimate for the complete sum gives $q^{-1/(2+\epsilon)}$ in place of $q^{-1/(k+\epsilon)}$.
\end{proof}

\section{The classical Fourier mass}\label{inifouma}

Suppose $A\subset[X]$ and $(A-A)\cap I(h_\ell)=\emptyset$. Also, set $\alpha:=|A|/X$. Suppose $\ell\le \exp\big((\log X)^{(2+\epsilon)/(2+2\epsilon)}\big)$  and let $\alpha$ satisfy
\[
\alpha\gg_{h} (\log X)^{2ek}\exp\!\Big(-(\log X)^{1/(2+2\epsilon)}\Big).
\]
Our goal in this section is to show that either $A$ enjoys one of the density increments in Proposition 
\ref{density-increment-calc} or $\widehat{\1_A}$ has $L^2$-Fourier mass on suitably chosen major arcs, following the classical ideas of S\'ark\"ozy. To this end, we must define the major and minor arcs. Recall that $L=\log(1/\alpha)$, and define 
\[
\tau:=C_1L^2 X^{-1},
\]
where $C_1$ is a positive constant to be determined later. Now, define the \emph{major arcs} $\mathfrak{M}\subseteq \R/\Z$ by 
\[
\mathfrak M
:=
\bigcup_{1\le q\le C_1\alpha^{-2-\epsilon}}
\ \bigcup_{\substack{0\le a<q\\(a,q)=1}}
\Big\{\xi\in\R/\Z:\ \Big\|\xi-\frac{a}{q}\Big\|_{\mathbb T}\le \tau\Big\}.
\]
The \emph{minor arcs} are given by $\mathfrak m=\mathfrak M^c$. 
The goal of this section is to prove the following.
\begin{lem}\label{InitialMassDichotomy}
Suppose $A\subset[X]$ and $(A-A)\cap I(h_\ell)=\emptyset$. Also, set $\alpha=|A|/X$ and suppose $\alpha\le2/3$. Suppose \eqref{ell-quasi-bound} holds and let $\alpha$ satisfy
\[
\alpha\gg_{h} (\log X)^{2ek}\exp\!\Big(-(\log X)^{1/(2+2\epsilon)}\Big),
\]
where we assume the implied constant is suitably large. Either there exists $\xi\in\R$ such that $|\xi|\le\tau$ and
\begin{equation}\label{InitialMass}
\sum_{2\leq q\leq C_1 \alpha^{-2-\epsilon}} q^{-1/(2+\epsilon)}
\sum_{\substack{1\leq a<q\\ (a,q)=1}}
|\widehat{\1_A}(a/q+\xi)|^2\gg_{h,\epsilon} \alpha^2X^2,
\end{equation}
or there is a progression $P\subseteq [X]$ of common difference $1$ and length $|P|\gg_{h,\epsilon} \frac{X}{L^2}$
on which the density of $A$ is at least $(1+c_{h,\epsilon})\alpha.$
\end{lem}

Before proving Lemma \ref{InitialMassDichotomy}, we need to prove two estimates. First, we have the following minor arc estimate, which in words says that $\widehat{g_{X,h_\ell}}$ is uniformly small on $\mathfrak{m}$. 

\begin{lem}[Minor arc estimate] \label{lem:minor_arc_estimate}
There exists a constant $F_h>0$ such that the following holds. Suppose \eqref{ell-quasi-bound} holds and let $0<\alpha\leq 2/3$ satisfy
\begin{equation}\label{eq:alphabig}
    \alpha\geq  F_h(\log X)^{2ek}\exp\!\Big(-(\log X)^{1/(2+2\epsilon)}\Big).
\end{equation}

Then, we have $\sup_{\xi\in \mathfrak m} |\widehat{g_{X,h_\ell}}(\xi)|\leq 2^{-9}\alpha X$.
\end{lem}

\begin{proof} 
Suppose $|\widehat{g_{X,h_\ell}}(\xi)|> 2^{-9}\alpha X.$ We show $\xi\in\mathfrak{M}$.
Since $J_{U,\ell}\ll_h(\log U)^{k-1}$ by Lemma \ref{J_USmall}, the lower bound on $\alpha$ allows us to apply Lemma \ref{Mukul} with $\delta=2^{-9}\alpha$, assuming we choose $F_h$ large enough. Thus there exists $q_0\in\N$ such that 
\[
q_0\ll_{h} \alpha^{-K}\ell^{k-1}X^{o_{h,\epsilon}(1)},
\qquad
\|q_0\xi\|_{\mathbb T}\ll_{h}\alpha^{-K-k}\ell^{2(k-1)}X^{-1+o_{h,\epsilon}(1)}.
\]
By the second estimate, there exist $a_0\in\Z$ and $\theta\in\R$ such that $\xi=\frac{a_0}{q_0}+\theta$ with
\[
|\theta|\ll_h\alpha^{-K-k}\ell^{2(k-1)}X^{-1+o_{h,\epsilon}(1)}.
\]
After reducing the fraction $a_0/q_0$, we may write it as $a/q$ with $(a,q)=1$ and $q\le q_0$. Using the lower bound assumption for $\alpha$, $\ell\le \exp\big((\log X)^{(2+\epsilon)/(2+2\epsilon)}\big)$, $Y=X^{1/k-o_{h,\epsilon}(1)}$, and $U=\exp((\log{X})^{1/2})$ we conclude
\begin{equation}\label{eq:qthetay}
    q<X^{o_{h}(1)},\qquad |\theta|\le X^{-1+o_{h,\epsilon}(1)},\qquad
\log(Y/q)\ge 2c\log U\log\log U,
\end{equation}

which are the conditions required to apply Lemma \ref{Rishika}. Thus the weaker conclusion of Lemma \ref{Rishika} gives
\[
2^{-9}\alpha X<\Bigl|\widehat{g_{X,h_\ell}}\Bigl(\frac{a}{q}+\theta\Bigr)\Bigr|
\ \ll_{h,\epsilon}\
Xe^{-\sqrt{|\theta|X/2}}q^{-1/(k+\epsilon)}
+
J_{U,\ell}Xe^{-c_2\frac{\log(Y/q)}{\log U}}.
\]
We bound the last term using Lemma \ref{J_USmall} and \eqref{eq:qthetay},  \[J_{U,\ell}Xe^{-c_2\frac{\log(Y/q)}{\log U}}\ll_{h,\epsilon} (\log{X})^{(k-1)/2}X\exp\left(\frac{-c_2}{2k}(\log{X})^{1/2}\right)=o(\alpha X),\]
where the last equality follows from \eqref{eq:alphabig}.
Hence 
\begin{align*}
    2^{-9}\alpha X&\ll_{h,\epsilon}Xe^{-\sqrt{|\theta|X/2}}q^{-1/(k+\epsilon)}\\
    \implies\alpha&\ll_{h,\epsilon} e^{-\sqrt{|\theta|X/2}}q^{-1/(k+\epsilon)}
\end{align*}

After rearranging, we have both
\[
|\theta|\ll_{h,\epsilon} L^2X^{-1}
\qquad\text{and}\qquad
q\ll_{h,\epsilon} \alpha^{-k-\epsilon}.
\]
By Lemma \ref{complete-denominator-bound}, $q^\sharp_{\ell,U}\ll_h q$. Since
\[
\alpha\gg_{h} (\log X)^{2ek}\exp\!\Big(-(\log X)^{1/(2+2\epsilon)}\Big)
\]
and $\frac{1}{2+2\epsilon}<\gamma$, it follows that $q^\sharp_{\ell,U}\le U$ for sufficiently large $X$. Therefore the stronger conclusion of Lemma \ref{Rishika} applies, giving
\[
\alpha \ll_{h,\epsilon} e^{-\sqrt{|\theta|X/2}}q^{-1/(2+\epsilon)}.
\]
Rearranging yields
\[
q\le C_1\alpha^{-2-\epsilon},
\]
provided $C_1$ is chosen sufficiently large. Since $|\theta|\ll_{h,\epsilon} L^2X^{-1}\le \tau$ once $C_1$ is large enough, we conclude that $\xi\in\mathfrak{M}$ as required. 
\end{proof} 


Next, we prove the following important physical space estimate before we prove Lemma \ref{InitialMassDichotomy}. 
\begin{lem}\label{physest}
Either 
\begin{equation}\label{Xsum}
\sum_{x,y}\1_A(x)\1_{[X]}(y)g_{X,h_\ell}(x-y)> 2^{-8}\alpha X^2,
\end{equation}
or there is an interval $P\subseteq[X]$ with $|P|\gg X$ on which the
density of $A$ is at least $5\alpha/4$.
\end{lem}

\begin{proof}
Suppose that
\begin{equation}\label{Xsum_fail}
\sum_{x,y}\1_A(x)\1_{[X]}(y)g_{X,h_\ell}(x-y)\le 2^{-8}\alpha X^2.
\end{equation}
We will show $A$ has a strong density increment. As usual, let $Y>0$ be defined by $h_\ell(Y)=X$.
Since $g_{X,h_\ell}\ge 0$, from \eqref{Xsum_fail} we may discard the contributions coming from
$y=x-h_\ell(n)$ and keep only those coming from $y=x+h_\ell(n)$, obtaining
\begin{align}
2^{-8}\alpha X^2
&\ge \sum_{x,y}\1_A(x)\1_{[X]}(y)g_{X,h_\ell}(x-y)\notag\\
&\ge J_{U,\ell}\sum_{x\in\mathbb Z}\sum_{n\in W_\ell(U)}\1_A(x)\1_{[X]}(x+h_\ell(n))h_\ell'(n)
w\Big(\frac{h_\ell(n)}{X}\Big).\label{orient_new}
\end{align}
Restricting to $1\le n\le Y$ and swapping the order of summation gives
\begin{equation}\label{swap_new}
J_{U,\ell}\sum_{\substack{1\le n\le Y\\n\in W_\ell(U)}} h_\ell'(n)w\Big(\frac{h_\ell(n)}{X}\Big)|A\cap[X-h_\ell(n)]|
\le 2^{-8}\alpha X^2.
\end{equation}
Using $\ell\le \exp\big((\log X)^{(2+\epsilon)/(2+2\epsilon)}\big)$, $h_\ell(1)\ll_h\ell^{k-1}=X^{o_{h,\epsilon}(1)}<X/4$ for $X$ sufficiently large. Since $h_\ell$ is continuous and strictly increasing on $[1,\infty)$ and $h_\ell(t)\to\infty$ as $t\to\infty$, there are unique $Y_1,Y_2>1$ such that
\[
h_\ell(Y_1)=\frac{X}{4},
\qquad
h_\ell(Y_2)=\frac{3X}{4}.
\]
Then $1/4<h_\ell(n)/X\leq3/4$ for $Y_1<n\leq Y_2$. Thus, the support properties of $w$ give $w(h_\ell(n)/X)\ge 1/2$ for $Y_1<n\leq Y_2$. Therefore \eqref{swap_new} implies
\begin{equation}\label{swap_restricted}
J_{U,\ell}\sum_{\substack{Y_1<n\leq Y_2\\n\in W_\ell(U)}} h_\ell'(n)|A\cap[X-h_\ell(n)]|\le 2^{-7}\alpha X^2.
\end{equation}
By Lemma~\ref{lem:invert_hl}, we have $Y_2-Y_1\asymp_hY$. Thus the short-interval range in Proposition~\ref{ricebrunsum2} holds with $q=1$, $v=Y_1$, and $H=Y_2-Y_1$. Moreover, Lemma~\ref{J_USmall} and the choice of $U$ show that when multiplied by the factor $J_{U,\ell}$, the resulting sieve error is $o_{h,\epsilon}(X)$. Applying the proposition and using the definition of $J_{U,\ell}$ gives
\begin{equation}\label{W_int_short}
J_{U,\ell}\sum_{\substack{Y_1<n\leq Y_2\\n\in W_\ell(U)}} h_\ell'(n)
=\frac{3X}{4}-\frac{X}{4}+o(X)=\frac X2+o(X).
\end{equation}
Combining \eqref{swap_restricted} and \eqref{W_int_short}, the pigeonhole principle gives $n\in (Y_1,Y_2]\cap W_\ell(U)$ such that, for $X$ sufficiently large,
\begin{equation*}\label{small_A}
|A\cap[X-h_\ell(n)]|
\le \left(\frac1{64}+o(1)\right)\alpha X
\le \frac{\alpha X}{63}.
\end{equation*}
Since $|A|=\alpha X$, we get
\[
|A\cap(X-h_\ell(n),X]|
\ge\frac{62}{63}\alpha X
\ge\frac54\alpha h_\ell(n),
\]
where the final inequality follows from $h_\ell(n)\le3X/4$. Since $h_\ell(n)\ge X/4$, this is a stronger density increment on the interval $P=(X-h_\ell(n),X]$ of length $|P|\gg X$, which is better than option \textup{(2)} and \textup{($\ast$)} in Proposition \ref{density-increment-calc}. 
\end{proof}

Finally, we are ready to prove Lemma \ref{InitialMassDichotomy}. 
\begin{proof}[Proof of Lemma \ref{InitialMassDichotomy}]
Set $f_A:=\1_A-\alpha\1_{[X]}$.
The condition $(A-A)\cap I(h_\ell)=\emptyset$, together with the definition of $g_{X,h_\ell}$ gives
\begin{equation*}\label{Asum}
\sum_{x,y}\1_A(x)\1_A(y)g_{X,h_\ell}(x-y)=0.
\end{equation*}
Using the triangle inequality and Lemma \ref{physest}, if the density increment alternative in Lemma \ref{physest} does not hold we have 
\[
\left|\sum_{x,y}\1_A(x)f_{A}(y)g_{X,h_\ell}(x-y)\right|\geq 2^{-8} \alpha^2 X^2.
\]
Using Fourier inversion, we obtain
\begin{equation}\label{fourier-correlation}
\left|\int_0^1\widehat{\1_A}(\xi)\widehat{f_A}(-\xi)\widehat{g_{X,h_\ell}}(\xi)\,d\xi\right|\geq  2^{-8} \alpha^2 X^2.
\end{equation}
By Cauchy--Schwarz, Parseval, and Lemma \ref{lem:minor_arc_estimate}, we have
\[
\left|\int_{\mathfrak m} \widehat{\1_A}(\xi)\widehat{f_A}(-\xi)\widehat{g_{X,h_\ell}}(\xi)\,d\xi\right|
\le \sup_{\xi\in\mathfrak m}\big|\widehat{g_{X,h_\ell}}(\xi)\big|
\left(\int_{\mathbb T}\big|\widehat{\1_A}(\xi)\big|^2\,d\xi\right)^{1/2}
\left(\int_{\mathbb T}\big|\widehat{f_A}(\xi)\big|^2\,d\xi\right)^{1/2}
\le 2^{-9}\alpha^2 X^2.
\]
Combining this with \eqref{fourier-correlation} yields
\begin{equation}\label{eq:major-arcs-mass}
\left|\int_{\mathfrak M} \widehat{\1_A}(\xi) \widehat{f_A}(-\xi) \widehat{g_{X,h_\ell}}(\xi)\,d\xi\right|
\ge 2^{-9}\alpha^2 X^2.
\end{equation}

Write
\[
\mathfrak M = [-\tau,\tau]\ \cup\ \mathfrak M',
\qquad
\mathfrak M' := \bigcup_{\substack{2\le q\le C_1\alpha^{-2-\epsilon}}}\ 
\bigcup_{\substack{1\le a<q\\(a,q)=1}}
\left[\frac{a}{q}-\tau,\ \frac{a}{q}+\tau\right],
\]
where $\tau=C_1L^2X^{-1}$ is as in the definition of the major arcs. Note here that each interval in the definition of $\mathfrak M$ is disjoint. 

From \eqref{eq:major-arcs-mass} we have either
\begin{equation}\label{eq:case1}
\left|\int_{-\tau}^{\tau} \widehat{\1_A}(\xi)\widehat{f_A}(-\xi)\widehat{g_{X,h_\ell}}(\xi)\,d\xi\right|
\ge 2^{-10}\alpha^2X^2,
\end{equation}
or
\begin{equation}\label{eq:case2}
\left|\int_{\mathfrak M'} \widehat{\1_A}(\xi)\widehat{f_A}(-\xi)\widehat{g_{X,h_\ell}}(\xi)\,d\xi\right|
\ge 2^{-10}\alpha^2X^2.
\end{equation}
We handle these two cases differently.

\medskip
\noindent\emph{Case 1: \eqref{eq:case1} holds.}
 Applying the triangle inequality to \eqref{eq:case1}, we get
\[
2^{-10}\alpha^2X^2
\le
\int_{-\tau}^{\tau}
\bigl|\widehat{\1_A}(\xi)\bigr|
\bigl|\widehat{f_A}(-\xi)\bigr|
\bigl|\widehat{g_{X,h_\ell}}(\xi)\bigr|\,d\xi.
\]
Using the trivial bound $|\widehat{\1_A}(\xi)|\le |A|=\alpha X$, we obtain
\[
2^{-10}\alpha^2X^2
\le
\alpha X\Bigl(\sup_{|\xi|\le \tau}\bigl|\widehat{f_A}(\xi)\bigr|\Bigr)
\int_{-\tau}^{\tau}\bigl|\widehat{g_{X,h_\ell}}(\xi)\bigr|\,d\xi,
\]
and hence
\begin{equation}\label{eq:case1-sup}
\Bigl(\sup_{|\xi|\le \tau}\bigl|\widehat{f_A}(\xi)\bigr|\Bigr)
\int_{-\tau}^{\tau}\bigl|\widehat{g_{X,h_\ell}}(\xi)\bigr|\,d\xi
\ \ge\ 2^{-10}\alpha X.
\end{equation}

We now bound the integral. By Lemma \ref{Rishika} with $(a,q)=(0,1)$ and $\theta=\xi$, for $X$ large the conditions to apply Lemma \ref{Rishika} are satisfied due to the definition of $\tau$ and the assumed lower bound on $\alpha$. Thus for all $|\xi|\le \tau$,
\[
\bigl|\widehat{g_{X,h_\ell}}(\xi)\bigr|
\ll_{h,\epsilon}
X e^{-\sqrt{|\xi|X/2}} +
J_{U,\ell}Xe^{-c_2\frac{\log Y}{\log U}}.
\]
Integrating and using the definition of $\tau=C_1L^2X^{-1}$ gives
\[
\int_{-\tau}^{\tau}\bigl|\widehat{g_{X,h_\ell}}(\xi)\bigr|d\xi
\ll_{h,\epsilon}\int_{-\tau}^{\tau}X e^{-\sqrt{|\xi|X/2}\,}d\xi +C_1L^2J_{U,\ell}e^{-c_2\frac{\log Y}{\log U}}\leq\int_{\R} e^{-\sqrt{|\xi|/2}}\,d\xi+o_{h,\epsilon}(1)=O(1) .
\]
Returning to \eqref{eq:case1-sup}, we conclude that
\[
\sup_{|\xi|\le \tau}\bigl|\widehat{f_A}(\xi)\bigr|
\gg_{h,\epsilon}
\alpha X.
\]
Choose $\xi_0\in[-\tau,\tau]$ such that
\[
\big|\widehat{f_A}(\xi_0)\big|\gg_{h,\epsilon}\alpha X.
\]
Applying Lemma~\ref{increment} with $q=1$, $a=0$, $\xi=\xi_0$, and a
sufficiently small fixed $\eta>0$, we obtain an interval
$P\subseteq[X]$ of length $|P|\gg T^{-1}X$ on which the density of $A$ is at least
$(1+c_{h,\epsilon})\alpha$. Since
\[
T=\max(1,|\xi_0|X)\ll L^2,
\]
this interval has length
\[
|P|\gg_{h,\epsilon}\frac{X}{L^2}.
\]
This is the second alternative of Lemma~\ref{InitialMassDichotomy}.

\medskip
\noindent\emph{Case 2: \eqref{eq:case2} holds.}
To handle this case, it is convenient to replace the balanced function $f_A$ by $\1_A$ on $\mathfrak M'$ in \eqref{eq:case2}. To be able to do so, we need to bound 
\begin{align*}
&\left|\int_{\mathfrak M'}\widehat{\1_A}(\xi)\widehat{f_A}(-\xi)\widehat{g_{X,h_\ell}}(\xi)\,d\xi
-\int_{\mathfrak M'}\widehat{\1_A}(\xi)\widehat{\1_A}(-\xi)\widehat{g_{X,h_\ell}}(\xi)\,d\xi\right| \\
&\qquad\le
\sup_{\xi\in\mathfrak M'}\big|\widehat{f_A}(\xi)-\widehat{\1_A}(\xi)\big|
\int_{\mathbb T}\big|\widehat{\1_A}(\xi)\big| \big|\widehat{g_{X,h_\ell}}(\xi)\big|\,d\xi,
\end{align*}
using the triangle inequality and extending the $\xi$-integral to $\mathbb T$.

Since $f_A=\1_A-\alpha\1_{[X]}$, we have
\begin{equation} \label{swap}
\sup_{\xi\in\mathfrak M'}\bigl|\widehat{f_A}(\xi)-\widehat{\1_A}(\xi)\bigr|
=\alpha \sup_{\xi\in\mathfrak M'}\bigl|\widehat{\1_{[X]}}(\xi)\bigr|\ll \alpha \sup_{\xi\in\mathfrak M'}\frac{1}{\|\xi\|_{\mathbb T}}.
\end{equation}
Since $\xi\in\mathfrak M'$, by definition there exist coprime integers $(a,q)=1$ with
$2\le q\le C_1\alpha^{-2-\epsilon}$ and some $\eta\in[-\tau,\tau]$ such that $\xi=\frac aq+\eta.$ 

Moreover, for $X$ sufficiently large we have $q\tau\le \tfrac12$, and hence
\begin{equation*}
    \|\xi\|_{\mathbb T}\ge \Bigl\|\frac aq\Bigr\|_{\mathbb T}-\|\eta\|_{\mathbb T}
\ge \frac1q-\tau \ge \frac1{2q}\gg \alpha^{2+\epsilon}.
\end{equation*}

Plugging this bound into \eqref{swap}, we get
\begin{equation}\label{eq: balanced ind diff}
    \sup_{\xi\in\mathfrak M'}\bigl|\widehat{f_A}(\xi)-\widehat{\1_A}(\xi)\bigr|
\ll \alpha^{-1-\epsilon}=X^{o_{h,\epsilon}(1)}.
\end{equation}

By Cauchy--Schwarz and Parseval,
\[
\int_{\mathbb T}\big|\widehat{\1_A}(\xi)\big| \big|\widehat{g_{X,h_\ell}}(\xi)\big|\,d\xi
\le
\left(\int_{\mathbb T}\big|\widehat{\1_A}(\xi)\big|^2\,d\xi\right)^{1/2}
\left(\int_{\mathbb T}\big|\widehat{g_{X,h_\ell}}(\xi)\big|^2\,d\xi\right)^{1/2}
=|A|^{1/2}\|g_{X,h_\ell}\|_2.
\]
Since $J_{U,\ell}=X^{o_h(1)}$ and $Y=X^{1/k-o_{h,\epsilon}(1)}$, using \eqref{eq:hl' bound} gives
\[
\|g_{X,h_\ell}\|_2^2=2J_{U,\ell}^2
\sum_{\substack{1\le n\le Y\\ n\in W_\ell(U)}}
h_\ell'(n)^2
w\left(\frac{h_\ell(n)}{X}\right)^2 \le 4J_{U,\ell}^2Yh_{\ell}'(Y)^2\ll_h J_{U,\ell}^2Y\left(\frac{X}{Y}\right)^2=X^{2-1/k+o_{h,\epsilon}(1)}.
\]
Consequently the error we get by replacing $f_A$ by $\1_A$  is
\[
\ll X^{o_h(1)}(\alpha X)^{1/2}X^{1-\frac1{2k}+o_{h,\epsilon}(1)}
=o_{h,\epsilon}(\alpha^2X^2),
\]
using the lower bound assumption on $\alpha$ in Lemma \ref{InitialMassDichotomy}. Thus, from \eqref{eq:case2} we may replace $f_A$ by $\1_A$ on $\mathfrak M'$ at negligible cost, obtaining
\[
\left|\int_{\mathfrak M'}\widehat{\1_A}(\xi)\widehat{\1_A}(-\xi)\widehat{g_{X,h_\ell}}(\xi)\,d\xi\right|
\ge 2^{-11}\alpha^2X^2.
\]
This implies
\begin{equation}\label{eq:mass-square}
\int_{\mathfrak M'} \big|\widehat{\1_A}(\xi)\big|^2 \big|\widehat{g_{X,h_\ell}}(\xi)\big|\,d\xi
\ge 2^{-11}\alpha^2X^2.
\end{equation}

We now estimate $|\widehat{g_{X,h_\ell}}(\xi)|$ on each major arc interval using Lemma \ref{Rishika}. Put $Q_0:=C_1\alpha^{-2-\epsilon}.$
By our lower bound for $\alpha$ and the choice of $\gamma$, we have $Q_0=o_{h,\epsilon}(U)$ for $X$ sufficiently large. Moreover, if $2\le q\le Q_0$, then $q^\sharp_{\ell,U}\ll_h q$ by Lemma \ref{complete-denominator-bound}, and hence $q^\sharp_{\ell,U}\le U$ for $X$ sufficiently large. Thus the stronger conclusion of Lemma \ref{Rishika} applies on all these arcs. 

Splitting $\mathfrak M'$ into the intervals
\[
\left[\frac aq-\tau,\frac aq+\tau\right]
\]
and writing $\xi=\frac aq+\eta$ with $|\eta|\le\tau$, we obtain
\begin{align*}
&\int_{-\tau}^{\tau}\sum_{\substack{2\le q\le C_1\alpha^{-2-\epsilon}}}
\sum_{\substack{1\le a<q\\(a,q)=1}}
\Bigl|\widehat{\1_A}\Bigl(\frac{a}{q}+\eta\Bigr)\Bigr|^2
\Bigl(q^{-1/(2+\epsilon)}X e^{-\sqrt{|\eta|X/2}}\Bigr)\,d\eta\\
&\qquad
+\int_{-\tau}^{\tau}\sum_{\substack{2\le q\le C_1\alpha^{-2-\epsilon}}}
\sum_{\substack{1\le a<q\\(a,q)=1}}
\Bigl|\widehat{\1_A}\Bigl(\frac{a}{q}+\eta\Bigr)\Bigr|^2
J_{U,\ell}Xe^{-c_2\frac{\log(Y/q)}{\log U}}\,d\eta
\ \gg\ \alpha^2X^2.
\end{align*}

The contribution of the sieve remainder is negligible. Indeed, we have, uniformly for
$2\le q\le C_1\alpha^{-2-\epsilon}$ and $|\eta|\le \tau$, that $Y/q=X^{1/k-o_{h,\epsilon}(1)}$, $\log U=(\log X)^{1/2}$, and
\[
J_{U,\ell}Xe^{-c_2\frac{\log(Y/q)}{\log U}}
\ll_h(\log X)^{(k-1)\gamma}X
\exp\left(-c_2\left(\frac{1}{k}-o_{h,\epsilon}(1)\right)(\log X)^{1/2}\right)
=o_{h,\epsilon}(\alpha X).
\]

Hence using Parseval, its total contribution is 
\[o_h\left(\alpha X\int_{\mathbb T}\big|\widehat{\1_A}(\xi)\big|^2\,d\xi\right)=o_{h,\epsilon}(\alpha^2X^2).\]
Therefore
\[
\int_{-\tau}^{\tau}\sum_{\substack{2\le q\le C_1\alpha^{-2-\epsilon}}}
\sum_{\substack{1\le a<q\\(a,q)=1}}
\Bigl|\widehat{\1_A}\Bigl(\frac{a}{q}+\eta\Bigr)\Bigr|^2
q^{-1/(2+\epsilon)}X e^{-\sqrt{|\eta|X/2}}\,d\eta
\gg_{h,\epsilon}\alpha^2X^2.
\]
Finally, since
\[
\int_{-\tau}^{\tau}X e^{-\sqrt{|\eta|X/2}\,}d\eta\le \int_{\R} e^{-\sqrt{|\eta|/2}}\,d\eta\ll 1,
\]
it follows that there exists some $\xi$ with $|\xi|\le \tau$ such that 
\[\sum_{2\leq q\leq C_1 \alpha^{-2-\epsilon}} q^{-1/(2+\epsilon)}
\sum_{\substack{1\leq a<q\\ (a,q)=1}}
|\widehat{\1_A}(a/q+\xi)|^2\gg_{h,\epsilon}\alpha^2X^2.\]
This proves the lemma. 
\end{proof}

\section{The classical density increment}\label{classical-density-increment}

We now begin the proof of Proposition~\ref{density-increment-calc}. In
Sections~\ref{classical-density-increment} and~\ref{applylevd}, the constants
$c_h,C_h,D_1,d_2$ are chosen by the following convention: $c_h$ and $d_2$ can be decreased, $C_h$ and $D_1$ can be increased. At the end of the proof these constants are fixed, and
depend only on $h$ and $\epsilon$.

Throughout this section, suppose $A\subset[X]$ and
$(A-A)\cap I(h_\ell)=\emptyset$. Set $\alpha=|A|/X$, and assume
\eqref{ell-quasi-bound}. Also assume $c_h<\alpha\leq 2/3 \text{ and }X>C_h.$

The goal of this section is to establish the \textup{($\ast$)} increment in
Proposition~\ref{density-increment-calc} in this range.

\medskip
We assume throughout this section that $\alpha>c_h$ and  
\begin{equation}\label{eq:alpha-lower-rho}
\alpha\gg_h (\log X)^{2ek}
\exp\!\Big(-(\log X)^{1/(2+2\epsilon)}\Big),
\end{equation}
after increasing $C_h$ if necessary. Thus Lemma~\ref{InitialMassDichotomy} is
available throughout this section.

\medskip

We now apply Lemma~\ref{InitialMassDichotomy}. If its density-increment alternative
occurs, then it is already stronger than \textup{($\ast$)}. We may therefore assume
that the Fourier-mass alternative in Lemma~\ref{InitialMassDichotomy} holds.
Then there exists $\xi\in\mathbb R$ with $|\xi|\le \tau=C_1L^{2}/X$ such that
\begin{equation}\label{eq:mass}
\sum_{2\le q\le C_{1}\alpha^{-2-\epsilon}} \frac1{q^{1/(2+\epsilon)}}
\sum_{\substack{1\le a<q\\(a,q)=1}}
\Bigl|\widehat{\1_A}\Bigl(\frac aq+\xi\Bigr)\Bigr|^{2}
\gg_{h,\epsilon}\alpha^{2}X^{2}.
\end{equation}
Set $Q:=C_1\alpha^{-2-\epsilon}$ and
\[
M:=\max_{\substack{2\le q\le Q\\ 1\le a<q\\(a,q)=1}}
\Bigl|\widehat{\1_A}\Bigl(\frac aq+\xi\Bigr)\Bigr|.
\]
We bound the total weight by
\[
\sum_{2\le q\le Q}\frac{1}{q^{1/(2+\epsilon)}}\sum_{\substack{1\le a<q\\(a,q)=1}}1
\le \sum_{2\le q\le Q} q \leq Q^{2}.
\]
Consequently, \eqref{eq:mass} implies $M^{2}\gg_{h,\epsilon} \alpha^{2}X^{2}/Q^{2}$, and hence there exist integers
$2\le q\le Q$ and $1\le a<q$ with $(a,q)=1$ such that
\begin{equation}\label{eq:large-fourier}
\Bigl|\widehat{\1_A}\Bigl(\frac aq+\xi\Bigr)\Bigr|
\gg_{h,\epsilon} \alpha^{3+\epsilon}X.
\end{equation}

We now pass from $\1_A$ to the balanced function $f_A$. By definition, $f_A:=\1_A-\alpha\1_{[X]}$. Using \eqref{eq: balanced ind diff}, the replacement error is $\ll \alpha^{-1-\epsilon}$. Using \eqref{eq:alpha-lower-rho}, we may enlarge $C_h$ so that this error term is negligible
compared to the right-hand side of \eqref{eq:large-fourier}.  Hence
\begin{equation}\label{eq:large-fourier-fa}
\Bigl|\widehat{f_A}\Bigl(\frac aq+\xi\Bigr)\Bigr|
\gg_{h,\epsilon}\alpha^{3+\epsilon}X.
\end{equation}

\medskip

Set $\eta:=20c\alpha^{2+\epsilon}$, where $c>0$ is chosen so that $\eta \in(0,1)$ and \eqref{eq:large-fourier-fa} implies $\Bigl|\widehat{f_A}\Bigl(\frac aq+\xi\Bigr)\Bigr|\ge \eta\alpha X.$
We may apply Lemma~\ref{increment} to
\eqref{eq:large-fourier-fa}, with the same $a$, $q$, and $\xi$.
Writing $T:=\max(1,|\xi|X),$
we obtain an arithmetic progression $P\subseteq [X]$
with common difference $\lambda(q)$ such that
\begin{equation}\label{eq:star-density}
\frac{|A\cap P|}{|P|}\ge \Bigl(1+\frac{\eta}{20}\Bigr)\alpha = \alpha + c\alpha^{3+\epsilon}
\end{equation}
and
\begin{equation}\label{eq:star-length}
|P|\gg_{h,\epsilon} \frac{\eta X}{T\lambda(q)}\gg_{h,\epsilon}\frac{\alpha^{2+\epsilon}X}{T\lambda(q)}.
\end{equation}
Using $\lambda(q)\le q^{k}$, $q\le C_1\alpha^{-2-\epsilon}$, and $T\ll L^{2}$ (since $|\xi|\le C_1L^{2}/X$), we get
\[
|P|
\gg \frac{\alpha^{2+\epsilon}X}{L^2\alpha^{-(2+\epsilon)k}}
=\frac{\alpha^{k(2+\epsilon)+2+\epsilon}X}{L^{2}}.
\]
Finally, since $0<\alpha\leq 2/3$, we have $L^{2}=\log^2(1/\alpha)< \alpha^{-2}$, giving
\begin{equation}\label{eq:star-length-clean}
|P|\gg \alpha^{k(2+\epsilon)+4+\epsilon}X.
\end{equation}

\medskip

Increasing $C_h$ if necessary so that
$C_h\ge k(2+\epsilon)+4+\epsilon$ and so that the fixed implicit constants in
\eqref{eq:star-density} and \eqref{eq:star-length-clean} are absorbed, the
bounds \eqref{eq:star-density} and \eqref{eq:star-length-clean} give
\[
|P|\ge \alpha^{C_h}X
\qquad\text{and}\qquad
\frac{|A\cap P|}{|P|}\ge \alpha+\alpha^{C_h}.
\]
This proves the \textup{($\ast$)} alternative in the range
$c_h<\alpha\le 2/3$.

\section{The random sparsification procedure}\label{applylevd}

In this section we complete the proof of Proposition
\ref{density-increment-calc}. The previous
section proves \textup{($\ast$)} when $c_h<\alpha\le2/3$. We therefore
assume that $0<\alpha\le c_h$ and that
\[
X\ge C_h,
\qquad
\ell\le
\exp\big((\log X)^{(2+\epsilon)/(2+2\epsilon)}\big).
\]

We shall apply Lemma~\ref{InitialMassDichotomy} with
$\epsilon':=\epsilon/6$ in place of $\epsilon$. 
Since $I(h_\ell)\subseteq h_\ell(\N)$, our assumption  $(A-A)\cap h_\ell(\N)\subseteq\{0\}$ implies
$(A-A)\cap I(h_\ell)=\emptyset$. 
In this section, we will replace $\epsilon$ with $\epsilon/6$ at various places to help with some calculations. The above bound on $\ell$ holds even after this replacement.

We will now apply Lemma \ref{InitialMassDichotomy}. If the lower-bound hypothesis on $\alpha$ in Lemma
\ref{InitialMassDichotomy} fails, then
\[
\alpha
\ll_{h,\epsilon}
(\log X)^{2ek}
\exp\!\Big(-\big(\log X\big)^{1/(2+\epsilon/3)}\Big)
\le
\exp\!\Big(-c_h(\log X)^{1/(2+2\epsilon)}\Big)
\]
for sufficiently large $X$, after decreasing $c_h$. This is conclusion
\textup{(1)} of Proposition
\ref{density-increment-calc}. Hence, for the rest of the argument we can assume \begin{equation} \label{eq:L bound}
    \alpha>\exp\!\Big(-c_h(\log X)^{1/(2+2\epsilon)}\Big) \text{ which is equivalent to } L\ll(\log X)^{1/(2+2\epsilon)}.
\end{equation}

If the density-increment alternative in Lemma \ref{InitialMassDichotomy} holds, then its progression has common difference $\lambda(1)=1$, length
$\gg X/L^2\gg X(\log X)^{-1/(1+\epsilon)}$, and a fixed multiplicative density gain. After adjusting
$D_1$ and $d_2$, this is conclusion \textup{(2)} with $j=1$ of Proposition
\ref{density-increment-calc}.

We may therefore suppose that there is $\xi\in\mathbb R$, with
$|\xi|\le C_1L^2/X$, such that
\begin{equation}\label{eq:leveld-input}
\sum_{2\le q\le C_1\alpha^{-2-\epsilon/6}}q^{-\sigma}
\sum_{\substack{1\le a<q\\(a,q)=1}}
\left|\widehat{\1_A}\left(\frac aq+\xi\right)\right|^2
\ge c_{\mathrm{mass}}\alpha^2X^2,
\qquad \sigma:=\frac1{2+\epsilon/6},
\end{equation}
where $c_{\mathrm{mass}}>0$ depends only on $h$ and $\epsilon$.

Put $B:=WL^{1+\epsilon}$, where $W$ is a sufficiently large constant depending only on $h$ and $\epsilon$. We decompose the denominators $q$ in \eqref{eq:leveld-input}. If $q=\prod_pp^{\nu_p}$, write $q=q_*q'$, where $q_*:=\prod_{\nu_p\ge3}p^{\nu_p}$ and $q':=\prod_{1\le\nu_p\le2}p^{\nu_p}$. Thus $q_*$ is cube-full, $(q_*,q')=1$, and
\begin{equation}\label{eq:cubefull-sum}
\sum_{\substack{t\ge1\\t\text{ cube-full}}}t^{-\sigma}\ll_\epsilon1,
\end{equation}
since $3\sigma>1$, as follows from the Euler product. 
We now define some notation used in this section.
\begin{itemize}
    \item Let $K(q')$ be one plus the number of distinct prime factors of $q'$ exceeding $B$.
    \item Put $d(q'):=\prod_{p\mid q',\,p\le B}p^{\nu_p}$. 
    \item List the remaining prime factors as $p_2(q')<\cdots<p_{K(q')}(q')$.
    \item Set $\mathcal D:=(\log L)^{-1}\mathbb Z_{\ge0}$. 
\end{itemize}

For a cube-full integer $t$, an integer $\kappa\ge1$, and $\vec\beta\in\mathcal D^\kappa$, let $\mathcal Q_{t,\kappa,\vec\beta}$ consist of the denominators $q$ in \eqref{eq:leveld-input} for which
\[q_*=t, \qquad K(q')=\kappa, \qquad d(q')\sim L^{\beta_1}, \qquad p_i(q')\sim L^{\beta_i} \text{ for } 2\le i\le\kappa,\]

 where $u\sim L^\beta$ means $L^\beta\le u<eL^\beta$. These classes partition the denominators. A nonempty class has $t\le C_1\alpha^{-2-\epsilon/6}$, $\kappa\le3L/\log L$, $\beta_1\ge0$, and $1+\epsilon\le\beta_2\le\cdots\le\beta_\kappa$, provided $L$ is sufficiently large. Define
\begin{align*}
E_{t,\kappa,\vec\beta}&:=\sum_{q\in\mathcal Q_{t,\kappa,\vec\beta}}\sum_{\substack{1\le a<q\\(a,q)=1}}\left|\widehat{\1_A}\left(\frac aq+\xi\right)\right|^2,\\
S_{t,\kappa,\vec\beta}&:=\sum_{q\in\mathcal Q_{t,\kappa,\vec\beta}}q^{-\sigma}\sum_{\substack{1\le a<q\\(a,q)=1}}\left|\widehat{\1_A}\left(\frac aq+\xi\right)\right|^2.
\end{align*}

We first rule out the case $\kappa\geq 2$. To this end, we establish
two complementary bounds, which play the role of the two endpoint
estimates in Green--Sawhney \cite[Lemma~7.1]{GreenSawhney2024FSv1}.

\begin{lem}\label{lem:leveld-endpoints}
Let $\kappa\ge2$. Then either conclusion \textup{(1)} or conclusion \textup{(2)} of Proposition \ref{density-increment-calc} holds, or
\begin{align}
E_{t,\kappa,\vec\beta}&\ll\left(\frac{C_0}{\kappa-1}\right)^{\kappa-1}L^{\kappa-1}\alpha^2X^2,\label{eq:E-r-one}\\
E_{t,\kappa,\vec\beta}&\ll\left(\frac{C_2(\kappa-1)}{W}\right)^{\kappa-1}L^{\beta_2+\cdots+\beta_\kappa-(1+\epsilon)(\kappa-1)}\alpha^2X^2,\label{eq:E-r-kappa}
\end{align}
where $C_2$ depends only on $h$ and $\epsilon$.
\end{lem}

\begin{proof}
We may assume that the class $\mathcal{Q}_{t,\kappa,\vec{\beta}}$ is nonempty and that neither conclusion of the proposition holds.

Put $R_1:=t\prod_{p\le B}p^2$. The prime number theorem gives $R_1\le\exp(CWL^{1+\epsilon})$. We can assume $R_1\leq X^{1/10}$. If not then $X^{1/10}<\exp(CWL^{1+\epsilon})$ which gives $L\gg(\log{X})^{1/(1+\epsilon)}$. This implies conclusion \textup{(1)} in Proposition \ref{density-increment-calc}.

Let $\mathcal R_2:=\{p^2:p>B,\ p\nmid t,\ p\le C_1\alpha^{-2-\epsilon/6}\}$. For any $R \in\N$ and $0\le b<R $, let 
\[A_{R ,b}:=\{m:b+R m\in A\}\text{ for } b\pmod{R}.\]
For $R_1\in\mathbb{N}$ and $0\le b<R_1,$ define $I_b:=\{m\in\mathbb Z:b+R_1m\in[X]\} \text{ and } N_b:=|I_b|$.

We apply Theorem \ref{level d}, with $d=\kappa-1$ and the pairwise coprime set $\mathcal R_2$, after translating each $I_b$ to $[N_b]$. Assumptions of Theorem \ref{level d} hold unless conclusion \textup{(1)} does: indeed $\kappa-1\le2^{-7}L$ follows from $\kappa\le3L/\log L$; if $N:=\lfloor X/R_1\rfloor$, then $\alpha>2N^{-1/2}$ unless $\alpha\ll X^{-9/20}$; and if some $s\in\mathcal R_2$ exceeds $N^{1/(32L)}$, then $s\le C_1^2\alpha^{-4-\epsilon/3}$ gives $L\gg_{h,\epsilon}(\log X)^{1/2}$. Each of these failures implies conclusion \textup{(1)} after decreasing $c_h$.

We use the following identity
\begin{equation}\label{eq:restriction-identity-endpoint}
\sum_{u\bmod R}\left|\widehat{\1_A}\left(\frac uR+\theta\right)\right|^2
=R\sum_{b\bmod R}\left|\widehat{\1_{A_{R,b}}}(R\theta)\right|^2,
\end{equation}
which follows from orthogonality. Indeed, we have
\begin{align*}
R\sum_{b\bmod R}\left|\widehat{\1_{A_{R,b}}}(R\theta)\right|^2&=R\sum_{b\bmod R}\,\sum_{m_1,m_2\in A_{R,b}}e((m_1-m_2)R\theta)\\
&=R\sum_{b\bmod R}\,\sum_{\substack{x_1,x_2\in A\\ x_1,x_2\equiv b\bmod R}}e((x_1-x_2)\theta)\\
&=R\sum_{\substack{x_1,x_2\in A\\x_1\equiv x_2\bmod R}}e((x_1-x_2)\theta)\\
&=R\sum_{x_1,x_2\in A}e((x_1-x_2)\theta)\,\frac{1}{R}\sum_{u\bmod R}e\left(\frac{(x_1-x_2)u}{R}\right)\\
&=\sum_{u\bmod R}\left|\widehat{\1_A}\left(\frac uR+\theta\right)\right|^2.
\end{align*}
Every $q\in\mathcal Q_{t,\kappa,\vec\beta}$ divides $R_1\prod_{i=2}^\kappa p_i(q')^2$. The Chinese remainder theorem therefore gives unique residues $u\bmod R_1$ and $v\bmod\prod_{i=2}^\kappa p_i(q')^2$ such that
\[
\frac aq\equiv\frac u{R_1}+\frac v{\prod_{i=2}^\kappa p_i(q')^2}\pmod1,
\qquad p_i(q')^2\nmid v\quad(2\le i\le\kappa).
\]
Here $v$ is divisible by $p_i(q')$ but not $p_i(q')^2$ when $p_i(q')\Vert q$, and is a unit modulo $p_i(q')^2$ when $p_i(q')^2\mid q$. The residues $u,v$ and the set $\{p_2(q')^2,\dots,p_\kappa(q')^2\}$ determine $a/q$, so this assignment is injective.

For each set $S=\{p_2(q')^2,\dots,p_\kappa(q')^2\}$ and each admissible $v\bmod\prod_{s\in S}s$, apply \eqref{eq:restriction-identity-endpoint} with $\theta=\xi+v/\prod_{s\in S}s$, and then replace $R_1v$ by $a$. Multiplication by $R_1$ permutes the residues not divisible by any $s\in S$ because $(R_1,s)=1$. Allowing all sets $S\subseteq\mathcal R_2$ of size $\kappa-1$ may add frequencies, and gives
\begin{equation}\label{eq:leveld-restriction}
E_{t,\kappa,\vec\beta}\le R_1\sum_{b\bmod R_1}\sum_{\substack{S\subseteq\mathcal R_2\\|S|=\kappa-1}}\sum_{\substack{a\md{\prod_{s\in S}s}\\s\in S\Rightarrow s\nmid a}}\left|\widehat{\1_{A_{R_1,b}}}\left(R_1\xi+\frac a{\prod_{s\in S}s}\right)\right|^2.
\end{equation}
For $f_{b}=\1_{A_{R_1,b}}(m_b+x)e(-R_1\xi x)$, where $m_b=\min{I_b}-1$, the sum over $S,a$ in \eqref{eq:leveld-restriction} is exactly the left side of \eqref{main41} in Theorem \ref{level d}, with $X=N_b$, $\mathcal Q=\mathcal R_2$, and $d=\kappa-1$.

If the first conclusion of Theorem \ref{level d} holds for every $b$, then $\sum_bN_b=X$ and $\max_bN_b\le X/R_1+1$ give
\[
R_1\alpha^2\sum_bN_b^2\left(\frac{C_0L}{\kappa-1}\right)^{\kappa-1}\ll\left(\frac{C_0}{\kappa-1}\right)^{\kappa-1}L^{\kappa-1}\alpha^2X^2,
\]
which proves \eqref{eq:E-r-one}. Otherwise, the second conclusion gives some $S\subseteq\mathcal R_2$, with $j:=|S|\le2L$, and a progression in $I_b$ of common difference $\prod_{s\in S}s$ and density greater than $2^j\alpha$. Lifting it to $[X]$ gives common difference $Q:=R_1\prod_{s\in S}s$ and length $X/Q+O(1)$. Since $Q\mid\lambda(Q)$, one of its subprogressions of common difference $\lambda(Q)$ has the same lower bound for the density and length $X/\lambda(Q)+O(1)$. Moreover, $\log Q\ll_{h,\epsilon}WL^{1+\epsilon}j$ and $\lambda(Q)\le Q^{\deg h}$. The failure of conclusion \textup{(1)} and $j\le2L$ give $\lambda(Q)=X^{o_{h,\epsilon}(1)}$, so the additive error is harmless; after increasing $D_1$, the length is at least $Xe^{-D_1L^{1+\epsilon}j}$. Taking $1+d_2\le2$ gives conclusion \textup{(2)}. This proves \eqref{eq:E-r-one} unless that conclusion already holds.
\medskip

We now prove \eqref{eq:E-r-kappa}. For $2\le i\le\kappa$, let $T_i$ be the primes in $[L^{\beta_i},eL^{\beta_i})$, put $x_i:=WL^{1+\epsilon}/(10(\kappa-1)\beta_i\log L)$, and set $M_i:=\max(1,\lfloor x_i\rfloor)$. Choose independently and uniformly a subset $\mathcal P_i\subseteq T_i$ of cardinality $M_i$. The prime number theorem gives $|T_i|\asymp L^{\beta_i}/(\beta_i\log L)$, hence:
\[
\mathbb P(p\in\mathcal P_i)\gg\frac W{\kappa-1}L^{1+\epsilon-\beta_i}.
\]
Independence and averaging therefore give choices of the $\mathcal P_i$ for which
\begin{equation}\label{eq:sparse-lower-endpoint}
\sum_{\substack{q\in\mathcal Q_{t,\kappa,\vec\beta}\\p_i(q')\in\mathcal P_i\ (2\le i\le\kappa)}}\sum_{\substack{ 1\leq a<q \\(a,q)=1}}\left|\widehat{\1_A}\left(\frac aq+\xi\right)\right|^2
\gg\left(\frac{W}{C_2(\kappa-1)}\right)^{\kappa-1}L^{(1+\epsilon)(\kappa-1)-(\beta_2+\cdots+\beta_\kappa)}E_{t,\kappa,\vec\beta}.
\end{equation}

Put $R_\kappa:=t(\prod_{p\le B}p^2)(\prod_{p\in\mathcal P_2\cup\cdots\cup\mathcal P_\kappa}p^2)$. Using the prime number theorem, the nonemptiness of $\mathcal{Q}_{t,\kappa, \vec{\beta}}$, the bound $M_i\leq 1+x_i$ and the definition of $x_i$, we obtain
\[\log{R_\kappa}\ll\log{t}+WL^{1+\epsilon}+WL^{1+\epsilon}\ll WL^{1+\epsilon}.\]

We can assume $R_\kappa\leq X^{1/10}$. If not then $X^{1/10}<\exp(CWL^{1+\epsilon})$ which gives $L\gg(\log{X})^{1/(1+\epsilon)}$. This implies conclusion \textup{(1)} in Proposition \ref{density-increment-calc}. 

Every denominator on the left of \eqref{eq:sparse-lower-endpoint} divides $R_\kappa$, and distinct reduced fractions give distinct residues modulo $R_\kappa$. By \eqref{eq:restriction-identity-endpoint}, the left-hand side of \eqref{eq:sparse-lower-endpoint} is at most
\[
R_\kappa\sum_{b\bmod R_\kappa}\left|\widehat{\1_{A_{R_\kappa,b}}}(R_\kappa\xi)\right|^2.
\]
If $|A_{R_\kappa,b}|\le2\alpha(X/R_\kappa+1)$ for every $b$, then the above quantity is at most $R_\kappa\max_b|A_{R_\kappa,b}|\sum_b|A_{R_\kappa,b}|\ll\alpha^2X^2$, which proves \eqref{eq:E-r-kappa}. Otherwise, for some $b$, $A$ has density greater than $2\alpha$ on the residue class $b\bmod R_\kappa$ in $[X]$. Refining this progression to common difference $\lambda(R_\kappa)$ gives one of density greater than $2\alpha$ and length $X/\lambda(R_\kappa)+O(1)$. Since $\log R_\kappa\ll_{h,\epsilon}WL^{1+\epsilon}$ and $\lambda(R_\kappa)\le R_\kappa^{\deg h}\ll_{h}\exp(O_{h,\epsilon}(\log{X})^{1/2})$, its length is at least $Xe^{-D_1L^{1+\epsilon}}$ after increasing $D_1$. Taking $1+d_2\le2$ gives conclusion \textup{(2)} with $j=1$. This proves the lemma.
\end{proof}

We may now assume that neither conclusion of the proposition holds and sum the lemma over the denominator classes. Put $\rho:=1/(2+\epsilon/2)$ and $S_0:=\beta_2+\cdots+\beta_\kappa$. Since $q\ge tL^{\beta_1+S_0}$, taking the geometric mean of \eqref{eq:E-r-one} and \eqref{eq:E-r-kappa} with exponents $1-\rho$ and $\rho$ gives
\begin{equation}\label{eq:S-endpoint-interpolation}
S_{t,\kappa,\vec\beta}\ll_{h,\epsilon}t^{-\sigma}(CW^{-\rho})^{\kappa-1}L^{-b(\kappa-1)-aS_0-\sigma\beta_1}\alpha^2X^2.
\end{equation}
Here $b:=\rho(2+\epsilon)-1=\epsilon/(4+\epsilon)>0$ and $a:=\sigma-\rho=4\epsilon/((12+\epsilon)(4+\epsilon))>0$; we also used $(\kappa-1)^{(2\rho-1)(\kappa-1)}\le1$.

Since $\mathcal D=(\log L)^{-1}\mathbb Z_{\ge0}$, we have $\sum_{\beta\in\mathcal D}L^{-a\beta}=(1-e^{-a})^{-1}$, and similarly with $a$ replaced by $\sigma$. Discarding the ordering and lower bounds on $\beta_2,\dots,\beta_\kappa$, and using \eqref{eq:cubefull-sum}, we obtain
\[
\sum_{\substack{t\text{ cube-full}\\\kappa\ge2,\ \vec\beta\in\mathcal D^\kappa}}S_{t,\kappa,\vec\beta}\ll_{h,\epsilon}\alpha^2X^2\sum_{\kappa\ge2}(C_3W^{-\rho}L^{-b})^{\kappa-1}\le\frac{c_{\mathrm{mass}}}{2}\alpha^2X^2
\]
after choosing $W$ sufficiently large. Thus at least half of the mass in \eqref{eq:leveld-input} comes from classes with $\kappa=1$.
Now, we handle the case with $\kappa=1$. For each cube-full $t$, define
\[
\mathcal E_t:=\sum_{\substack{2\le q\le C_1\alpha^{-2-\epsilon/6}\\q_*=t,\ K(q')=1}}\sum_{\substack{1\le a<q\\(a,q)=1}}\left|\widehat{\1_A}\left(\frac aq+\xi\right)\right|^2.
\]
Since $q^{-\sigma}\le t^{-\sigma}$ whenever $q_*=t$, we have 
\[\frac {c_{\mathrm{mass}}}{2}\alpha^2X^2\le\sum_{t\text{ cube-full}}t^{-\sigma}\mathcal E_t.\] Equation \eqref{eq:cubefull-sum} therefore shows that some cube-full $t$ satisfies $\mathcal E_t\gg_{h,\epsilon}\alpha^2X^2$.

Put $Q:=t\prod_{p\le B}p^2$. Every denominator contributing to $\mathcal E_t$ divides $Q$. Moreover, by \eqref{eq:L bound} $\log Q\ll_{h,\epsilon}WL^{1+\epsilon}\ll_{h,\epsilon}W(\log X)^{1/2}$. Distinct reduced fractions contributing to $\mathcal E_t$ give distinct nonzero residues modulo $Q$, and therefore
\begin{equation}\label{eq:kappa-one-common-denominator}
\sum_{\substack{b\bmod Q\\b\ne0}}\left|\widehat{\1_A}\left(\frac bQ+\xi\right)\right|^2\gg_{h,\epsilon}\alpha^2X^2.
\end{equation}

Recall that $f_A=\1_A-\alpha\1_{[X]}$. Since conclusion \textup{(1)} fails, $L^2=X^{o(1)}$, so $Q|\xi|=o(1)$ and $|\xi|\le1/(2Q)$. For $b\not\equiv0\pmod Q$, we therefore have $\|b/Q+\xi\|_{\mathbb T}\ge1/(2Q)$ and $|\widehat{\1_{[X]}}(b/Q+\xi)|\ll Q$. Summing $|z-w|^2\ge|z|^2/2-|w|^2$ over $b$, the total contribution of $\alpha\widehat{\1_{[X]}}$ is $O(\alpha^2Q^3)=o_{h,\epsilon}(\alpha^2X^2)$. Hence
\[
\sum_{b\bmod Q}\left|\widehat{f_A}\left(\frac bQ+\xi\right)\right|^2\gg_{h,\epsilon}\alpha^2X^2.
\]
Apply Lemma \ref{increment} with $q=Q$ and a sufficiently small constant $\eta>0$ depending only on $h$ and $\epsilon$. Its condition $3Q\le X$ holds, and $T=\max(1,|\xi|X)\ll L^2$. After decreasing $d_2$ if necessary, we obtain a progression $P\subseteq[X]$ of common difference $\lambda(Q)$, density at least $(1+d_2)\alpha$, and length $|P|\gg\lambda(Q)^{-1}L^{-2}X$. Finally $\lambda(Q)\le Q^{\deg h}\le\exp(CWL^{1+\epsilon})$, so $|P|\ge Xe^{-D_1L^{1+\epsilon}}$ after increasing $D_1$. This is conclusion \textup{(2)} with $j=1$, and the proof of Proposition \ref{density-increment-calc} is complete.

\begin{appendix}
\section{The sieve calculation}\label{sievecalc}

We use the fundamental lemma of sieve theory in the direct form~\cite[Corollary 19]{TaoSieveNotes}.
\begin{lem}[Fundamental lemma of sieve theory, direct form]\label{fundamental-sieve-lemma}
Let $\kappa,s,D\ge 1$. Let $z=D^{1/s}$, and let $g$ be a multiplicative function with
$0\le g(p)<1$ for every prime $p$. Put
\[
V(w):=\prod_{p<w}(1-g(p)).
\]
Assume that, for some $A\ge1$,
\begin{equation}\label{eq:sieve-dimension}
V(w)\le A\left(\frac{\log z}{\log w}\right)^\kappa V(z)
\qquad (2\le w\le z).
\end{equation}
For each prime $p$, let $E_p$ be a set of integers, and let
$(a_n)_{n\in\mathbb Z}$ be a finitely supported sequence of nonnegative real numbers.
For squarefree $d$, define $E_d:=\bigcap_{p\mid d}E_p.$
Suppose that for every squarefree $d\le D$, one has $\sum_{n\in E_d}a_n=Xg(d)+r_d$, where $X >0$ and $r_d \in \R$.
Then
\[
\sum_{n\notin \bigcup_{p<z}E_p}a_n
=
\bigl(1+O_{\kappa,A}(e^{-s})\bigr)XV(z)
+
O_{\kappa,A}\left(\sum_{d\le D}\mu(d)^2|r_d|\right).
\]
\end{lem}

With this at our disposal, we now prove Proposition \ref{ricebrunsum2}.
\begin{proof}[Proof of Proposition \ref{ricebrunsum2}]
Suppose first that $b\notin W_\ell^q(U)$. By the definition of
$W_\ell^q(U)$, there is a prime $p\le U$ such that
\[
p^{\gamma_\ell(p)}\mid q
\qquad\text{and}\qquad
h_\ell'(b)\equiv0\pmod{p^{\gamma_\ell(p)}}.
\]
Hence every $n\equiv b\pmod q$ satisfies
\[
h_\ell'(n)\equiv h_\ell'(b)\equiv0
\pmod{p^{\gamma_\ell(p)}},
\]
so no such $n$ lies in $W_\ell(U)$. The first claim follows. 
\medskip

We now fix $b\in W_\ell^q(U)$. 

For $v\le u\le v+H$, define
\[
A_{b,v}(u)
:=
\#\{v<n\le u:n\equiv b\pmod q,\ n\in W_\ell(U)\}.
\]
We first prove a uniform estimate
\begin{equation}\label{eq:brun-short-count}
A_{b,v}(u)=
\frac{u-v}{q}
\prod_{\substack{p\le U\\p^{\gamma_\ell(p)}\nmid q}}
\left(1-\frac{j_\ell(p)}{p^{\gamma_\ell(p)}}\right)
+
O_h\left(
\left(\frac Hq\right) e^{-c'\frac{\log (H/q)}{\log U}}
\right).
\end{equation}
The assertion is immediate when $u=v$, so suppose that $u>v$.

Write every integer congruent to $b$ modulo $q$ uniquely as $b+qm$. Let
\[
\mathcal I_{v,u}
:=
\{m\in\Z:v<b+qm\le u\},
\qquad
a_m:=\1_{\mathcal I_{v,u}}(m).
\]
The interval defining $\mathcal I_{v,u}$ has length $(u-v)/q$. This positive quantity will be the parameter denoted by $X$ in Lemma~\ref{fundamental-sieve-lemma}. The argument below therefore also applies when $\mathcal I_{v,u}$ is empty.

We now determine the sieving conditions. For every prime $p\le U$ with
$p^{\gamma_\ell(p)}\nmid q$, let $E_p$ be the set of integers $m$
satisfying
\[
h_\ell'(b+qm)\equiv0\pmod{p^{\gamma_\ell(p)}}.
\]
For all other primes $p$, put $E_p=\emptyset$. If
$p^{\gamma_\ell(p)}\mid q$, then $b\in W_\ell^q(U)$ gives
\[
h_\ell'(b+qm)\equiv h_\ell'(b)\not\equiv0
\pmod{p^{\gamma_\ell(p)}}
\]
for every $m$. It follows that
\[
b+qm\in W_\ell(U)
\qquad\Longleftrightarrow\qquad
m\notin\bigcup_{p<z}E_p,
\]
and hence for $z:=2U$
\[
A_{b,v}(u)
=
\sum_{m\notin\bigcup_{p<z}E_p}a_m.
\]

Define the multiplicative function $g$ appearing in Lemma
\ref{fundamental-sieve-lemma} by
\[
g(p):=
\begin{cases}
\dfrac{j_\ell(p)}{p^{\gamma_\ell(p)}},
&
p\le U
\text{ and }
p^{\gamma_\ell(p)}\nmid q,\\[6pt]
0,
&
\text{otherwise},
\end{cases}
\]
and extend $g$ multiplicatively. Note that $0\le g(p)<1$.

For squarefree $d$, put
\[
E_d:=\bigcap_{p\mid d}E_p.
\]
If $d$ has a prime factor for which $E_p=\emptyset$, then
$E_d=\emptyset$ and $g(d)=0$, so the counting formula required by the
fundamental lemma holds with remainder zero. Suppose now that $E_p$ is non empty for every prime factor $p|d$. In particular, $p\le U$ and
$p^{\gamma_\ell(p)}\nmid q$. Since $q$ is $(\ell,U)$-complete, this
implies $p\nmid q$ for every $p\mid d$. Consequently, for each
$p\mid d$, the map
\[
m\pmod{p^{\gamma_\ell(p)}}
\longmapsto
b+qm\pmod{p^{\gamma_\ell(p)}}
\]
is a bijection. By the definition of $j_\ell(p)$, the set $E_p$
therefore consists of exactly $j_\ell(p)$ residue classes modulo
$p^{\gamma_\ell(p)}$.

The moduli $p^{\gamma_\ell(p)}$ for $p\mid d$ are pairwise coprime, so
the Chinese Remainder Theorem shows that $E_d$ is a union of
\[
J_d:=\prod_{p\mid d}j_\ell(p)
\]
residue classes modulo
\[
R_d:=\prod_{p\mid d}p^{\gamma_\ell(p)}.
\]
Moreover, by the multiplicativity of $g$,
\[
g(d)
=
\prod_{p\mid d}\frac{j_\ell(p)}{p^{\gamma_\ell(p)}}
=
\frac{J_d}{R_d}.
\]
Thus $g(d)$ is precisely the proportion of residue classes modulo
$R_d$ which belong to $E_d$. Each of these $J_d$ residue classes
contains $\frac{u-v}{qR_d}+O(1)$ integers in $\mathcal I_{v,u}$.
Therefore
\begin{equation}\label{eq:local-count-d}
\sum_{m\in E_d}a_m
=
\frac{u-v}{q}\frac{J_d}{R_d}+r_d
=
\frac{u-v}{q}g(d)+r_d,
\qquad
|r_d|\le J_d.
\end{equation}
This verifies the counting formula required for the application of
the fundamental lemma.

We next bound $J_d$. If $p>k$ and $\gamma_\ell(p)=1$, then
$h_\ell'$ is a nonzero polynomial modulo $p$ of degree at most $k-1$,
so $j_\ell(p)\le k-1$. The primes $p\le k$ contribute only an
$O_h(1)$ factor. The primes for which $\gamma_\ell(p)>1$ also
contribute only an $O_h(1)$ factor. Indeed, \eqref{completersmall}
gives
\[
\prod_{\gamma_\ell(p)>1}p^{\gamma_\ell(p)-1}\ll_h1.
\]
Since $p\le p^{\gamma_\ell(p)-1}$ whenever
$\gamma_\ell(p)>1$, this also gives
\[
\prod_{\gamma_\ell(p)>1}p^{\gamma_\ell(p)}
\le
\left(
\prod_{\gamma_\ell(p)>1}p^{\gamma_\ell(p)-1}
\right)^2
\ll_h1.
\]
Because $j_\ell(p)<p^{\gamma_\ell(p)}$, it follows that $\displaystyle J_d\ll_h\prod_{\substack{\gamma_{\ell}(p)=1}\\p|d}j_{\ell}(p)\ll_{h}(k-1)^{\omega(d)}.$

We also verify the sieve-dimension condition. Recall that
\[
V(w):=\prod_{p<w}(1-g(p)).
\]
For $2\le w\le z=2U$, we have
\[
\frac{V(w)}{V(z)}
=
\prod_{w\le p<z}(1-g(p))^{-1}
\ll_h
\prod_{\substack{w\le p<z\\p>k}}
\left(1-\frac{k-1}{p}\right)^{-1}
\ll_k
\left(\frac{\log z}{\log w}\right)^{k-1}.
\]
The primes $p\le k$ and the primes satisfying
$\gamma_\ell(p)>1$ contribute to the implied constant. Thus the
dimension condition in Lemma~\ref{fundamental-sieve-lemma} holds with
$\kappa=k-1$ and $A=O_h(1)$.
Put
\[
D:=\left(\frac Hq\right)^{1/2},
\qquad
s:=\frac{\log D}{\log z}
=\frac{\log D}{\log(2U)}.
\]
By \eqref{eq:local-count-d}, the sum of the remainder terms in the
fundamental lemma satisfies
\begin{align*}
\sum_{d\le D}\mu(d)^2|r_d|
\ll_h
\sum_{d\le D}\mu(d)^2(k-1)^{\omega(d)}
\le
D\sum_{d\le D}
\frac{\mu(d)^2(k-1)^{\omega(d)}}d
&\le
D\prod_{p\le D}
\left(1+\frac{k-1}{p}\right)\\
&\ll_k
D(\log(2D))^{k-1}.
\end{align*}
After increasing $c$ and decreasing $c'$, this is
\begin{equation}\label{eq:short-sieve-remainder}
\ll_h
\frac Hq e^{-c'\frac{\log ( H/q)}{\log U}}.
\end{equation}

The relative error supplied by the fundamental lemma is also bounded
by \eqref{eq:short-sieve-remainder}. Indeed,
\[
e^{-s}
=
\exp\left(
-\frac{\log D}{\log(2U)}
\right)
\le
\exp\left(
-\frac{c_1\log(H/q)}{\log U}
\right)
\]
for some absolute constant $c_1>0$. Since $0<(u-v)/q\le H/q$ and $V(z)\le1,$
the relative error contributes at most $O_h\left(0
\frac Hq e^{-c'\frac{\log (H/q)}{\log U}}
\right)$
after decreasing $c'$ if necessary.

Lemma~\ref{fundamental-sieve-lemma} now gives
\[
A_{b,v}(u)
=
\frac{u-v}{q}V(z)
+
O_h\left(
\frac Hq e^{-c'\frac{\log(H/q)}{\log U}}
\right).
\]
Because $g(p)=0$ for $U<p<2U$, we have
\[
V(z)
=
\prod_{\substack{p\le U\\p^{\gamma_\ell(p)}\nmid q}}
\left(1-\frac{j_\ell(p)}{p^{\gamma_\ell(p)}}\right).
\]
This proves \eqref{eq:brun-short-count} uniformly for
$v\le u\le v+H$.

We now recover the weight $h_{\ell}'(n)$. Partial summation gives
\[
\sum_{\substack{v<n\le t\\
n\in W_\ell(U)\\
n\equiv b\bmod q}}
h_\ell'(n)
=
h_\ell'(t)A_{b,v}(t)
-
\int_v^t A_{b,v}(u)h_\ell''(u)\,du.
\]
Substituting \eqref{eq:brun-short-count}, the main term is
\[
\frac1q
\prod_{\substack{p\le U\\p^{\gamma_\ell(p)}\nmid q}}
\left(1-\frac{j_\ell(p)}{p^{\gamma_\ell(p)}}\right)\cdot
\left(
(t-v)h_\ell'(t)
-
\int_v^t(u-v)h_\ell''(u)\,du
\right).
\]
Integration by parts gives
\[
\int_v^t(u-v)h_\ell''(u)\,du
=
(t-v)h_\ell'(t)-h_\ell(t)+h_\ell(v).
\]
Thus the main term is exactly
\[
\frac{h_\ell(t)-h_\ell(v)}q
\prod_{\substack{p\le U\\p^{\gamma_\ell(p)}\nmid q}}
\left(1-\frac{j_\ell(p)}{p^{\gamma_\ell(p)}}\right).
\]

Since $h_\ell', h_{\ell}''>0$ on $[1,\infty)$, the error coming from \eqref{eq:brun-short-count} is uniform in $u$ and  is at most
\begin{align*}
   & \ll_h
\frac Hq
e^{-c'\frac{\log(H/q)}{\log U}}
\left(
h_\ell'(t)+\int_v^t h_\ell''(u)\,du
\right)\\
&\leq \frac Hq
e^{-c'\frac{\log(H/q)}{\log U}}(2h_\ell'(v+H)).
\end{align*}

This proves \eqref{eq:ricebrun-short-derivative}.

If $v+H\le Y$ and $h_\ell(Y)=X$, then, since $h_\ell''>0$ on $[1,\infty)$, \eqref{eq:hl' bound} gives $h_\ell'(v+H)\le h_\ell'(Y)\ll_h\frac XY$,
which yields the desired error term
\[
O_h\left(
\frac{XH}{qY}
e^{-c'\frac{\log(H/q)}{\log U}}
\right).
\]

\end{proof}
\end{appendix}

\section*{Acknowledgments}
The seventh and ninth authors were supported by the HUN-REN Alfr\'ed R\'enyi Institute of Mathematics (Erd\H{o}s Center), and they thank the R\'enyi Institute for its hospitality. The seventh author would like to thank Alex Rice for introducing him to the circle method and forbidden difference problems, and he would also like to thank Mehtaab Sawhney for pointing out to us the idea of applying Alex Rice's exponential sum estimates to improve an earlier version of Theorem \ref{main}.

\section*{AI disclosure}
During the preparation of the present version of this manuscript, after the first version had appeared on arXiv, the authors used an LLM in an effort to improve the admissible range of exponents in the main theorem. The LLM suggested ideas that the authors developed into the short-interval strategy used in the major-arc estimate and helped them identify the relevant precedent for this strategy in Vinogradov’s work; see \cite[Theorems~2a--2b]{Vinogradov54}. In addition, the authors had previously written an application of the random sparsification technique, but the LLM helped them identify the simpler form of the argument that now appears in the manuscript. The LLM was also used for LaTeX assistance, formatting, and copyediting. The authors wrote and verified the manuscript and take full responsibility for its contents.

\bibliographystyle{amsplain}
\bibliography{citations}
\end{document}